\documentclass[11pt, a4paper]{amsart} 

\title{Equations of tropical varieties} 

\usepackage[foot]{amsaddr}
\author{Jeffrey Giansiracusa$^1$} 
\email{j.h.giansiracusa@swansea.ac.uk} 
\address{$^1$Department of
  Mathematics, Swansea University \\ Singleton Park \\ Swansea SA2 8PP, UK} 

\author{Noah  Giansiracusa$^2$} 
\email{noahgian@uga.edu}
\address{$^2$Department of Mathematics\\
University of Georgia\\
Athens, GA 30602, USA}

\date{5 January, 2016}

\subjclass[2010]{14T05 (primary), 14A20 (secondary)}

\usepackage{mathptmx}
\DeclareSymbolFont{cmlargesymbols}{OMX}{cmex}{m}{n}
\DeclareMathSymbol{\mycoprod}{\mathop}{cmlargesymbols}{"60}

\usepackage{amssymb}
\usepackage{mathrsfs}  
\usepackage{latexsym}
\usepackage{amsmath}
\usepackage{color}

\usepackage[mathcal]{eucal}
\DeclareFontFamily{OT1}{pzc}{}
\DeclareFontShape{OT1}{pzc}{m}{it}{<-> s * [1.10] pzcmi7t}{}
\DeclareMathAlphabet{\mathpzc}{OT1}{pzc}{m}{it}

\usepackage[cmtip,arrow]{xy}
\usepackage{pb-diagram,pb-xy}
\usepackage[vmargin=3cm, hmargin=3cm]{geometry}
\numberwithin{equation}{subsection}

\newtheorem{thmA}{Theorem}

\newtheorem{theorem}{Theorem}[subsection]  

\newtheorem{lemma}[theorem]{Lemma} 
\newtheorem{proposition}[theorem]{Proposition}
\newtheorem{corollary}[theorem]{Corollary}

\theoremstyle{remark} 
\newtheorem{remark}[theorem]{Remark}
\newtheorem{definition}[theorem]{Definition}
\newtheorem{example}[theorem]{Example}

\newcommand{\Cox}{\mathrm{Cox}}
\newcommand{\Cl}{\mathrm{Cl}}

\newcommand{\Hom}{\mathrm{Hom}}

\newcommand{\im}{\operatorname{im}}
\newcommand{\id}{\mathrm{id}}
\newcommand{\F}{\mathbb{F}}
\newcommand{\B}{\mathbb{B}}
\newcommand{\Fun}{\mathbb{F}_1}
\newcommand{\C}{\mathbb{C}}
\newcommand{\R}{\mathbb{R}}
\newcommand{\N}{\mathbb{N}}
\newcommand{\Z}{\mathbb{Z}}

\newcommand{\PP}{\mathbb{P}}
\newcommand{\T}{\mathbb{T}}
\newcommand{\A}{\mathbb{A}_\T}
\newcommand{\spec}{\mathrm{Spec}\:}
\newcommand{\proj}{\mathrm{Proj}\:}
\newcommand{\Trop}{\mathpzc{Trop}}
\newcommand{\trop}{\mathpzc{trop}}

\newcommand{\val}{\mathpzc{Val}}
\newcommand{\bend}{\mathpzc{B}}
\newcommand{\Bend}{\mathpzc{Bend}}
\newcommand{\setBend}{\mathpzc{bend}}
\newcommand{\supp}{\mathrm{supp}}
\newcommand{\Sch}{\mathrm{Sch}}

\newcommand{\Pic}{\mathrm{Pic}}

\newcommand{\defeq}{\mathbin{\mathpalette{\vcenter{\hbox{$:$}}}=}}

\def\co{\colon\thinspace} 

\begin{document}

\begin{abstract}
  We introduce a scheme-theoretic enrichment of the principal objects of tropical geometry.  Using a
  category of semiring schemes, we construct tropical hypersurfaces as schemes over idempotent
  semirings such as $\T = (\mathbb{R}\cup \{-\infty\}, \mathrm{max}, +)$ by realizing them as
  solution sets to explicit systems of tropical equations that are uniquely determined by idempotent module theory.  We then define a tropicalization functor that sends closed subschemes of a toric
  variety over a ring $R$ with non-archimedean valuation to closed subschemes of the corresponding
  tropical toric variety.  Upon passing to the set of $\T$-points this reduces to Kajiwara-Payne's
  extended tropicalization, and in the case of a projective hypersurface we show that the scheme
  structure determines the multiplicities attached to the top-dimensional cells.  By varying the
  valuation, these tropicalizations form algebraic families of $\T$-schemes parameterized by a
  moduli space of valuations on $R$ that we construct.  For projective subschemes, the Hilbert polynomial is preserved by
  tropicalization, regardless of the valuation. We conclude with some examples and a discussion of
  tropical bases in the scheme-theoretic setting.
\end{abstract}

\maketitle
\tableofcontents

\begin{center}\emph{Dedicated to Max and Add(ie)}\end{center}

\section{Introduction}
Tropical geometry is an emerging tool in algebraic geometry that can transform certain questions into
combinatorial problems by replacing a variety with a polyhedral object called a tropical variety.
It has had striking applications to a range of subjects, such as enumerative geometry
\cite{Mikhalkin-P2, Fomin-Mikhalkin, Gathmann-Markwig,Ardila-Block}, classical geometry \cite{Brill-Noether, Baker},
intersection theory \cite{Katz,Gibney-Maclagan,Osserman-Payne}, moduli spaces and compactifications
\cite{Tevelev,Hacking-Keel-Tevelev,Abramovich-Caporaso-Payne, Sturmfels-Sam-Ren}, mirror symmetry
\cite{Gross-P2,Tropical-vertex,Gross-tropical}, abelian varieties \cite{Gubler,Caporaso-Viviani},
representation theory \cite{Cluster-I,Russian-rep}, algebraic statistics and mathematical biology
\cite{Pachter-Sturmfels, Manon} (and many more papers by many more authors).  Since its inception,
it has been tempting to look for algebraic foundations of tropical geometry, e.g., to view tropical
varieties as varieties in a more literal sense and to understand tropicalization as a degeneration
taking place in one common algebro-geometric world.  However, tropical geometry is based on the
idempotent semiring $\T = (\R\cup \{-\infty\},\max,+)$, which is an object outside the traditional
scope of algebraic geometry.

Motivated by the desire to do algebraic geometry over the \emph{field with one element}, $\Fun$,
various authors have constructed extensions of Grothendieck's scheme theory to accommodate geometric
objects whose functions form algebraic objects outside the category of rings, such as semirings and
monoids---the context of $\Fun$-geometry.  The three theories developed in \cite{Durov,Toen-Vaquie,
  Lorscheid-blueprints1} essentially coincide over semirings, where the resulting schemes can be
described in familiar terms either as spaces equipped with a sheaf of semirings, or as functors of
points extended from rings to the larger category of semirings.  Although these three theories
produce distinct categories of $\Fun$-schemes, there is a smaller category of naive $\Fun$-schemes
(containing the category of split toric varieties with torus-equivariant morphisms) that embeds as a
full subcategory of each and is appropriate for the purposes of this paper; moreover, there are
base-change functors from (each version of) $\Fun$-schemes to schemes over any ring or semiring.
The above-cited authors have each speculated that the category of schemes over the semiring $\T$
might provide a useful foundation for tropical geometry.  However, tropicalization, as it currently
exists in the literature, produces tropical varieties as sets rather than as solutions to systems of
tropical equations.  Since the set of geometric points of a scheme is very far from determining the
scheme, the challenge is to lift tropicalization to schemes in an appropriate way.

In traditional tropical geometry (e.g., \cite{MS}) one considers subvarieties of a torus defined
over a non-archimedean valued field $k$; usually one assumes the valuation is nontrivial and the
field is algebraically closed and complete with respect to the valuation.  Tropicalization sends a
subvariety $Z$ of the torus $(k^\times)^n$ to the polyhedral subset $\trop(Z)$ of the tropical torus
$(\T^\times)^n = \R^n$ constructed as the Euclidean closure of the image of coordinate-wise
valuation.  Kajiwara and Payne extended this set-theoretic tropicalization to subvarieties of a
toric variety by using the stratification by torus orbits \cite{Kajiwara,Payne}.  A fan determines a
toric scheme $X$ over $\Fun$ and base-change to $k$ yields a familiar toric variety $X_k$, while
base-change to $\T$ yields a tropical toric scheme $X_\T$. The $\T$-points of $X_\T$ (or
equivalently, of $X$) form the partial compactification of $N_\R$ dual to the fan, and the
Kajiwara-Payne tropicalization map $\trop$ sends subvarieties of $X_k$ to subsets of $X(\T)$.  We
give a scheme-theoretic enhancement, $\Trop$, of this process.

\begin{thmA}\label{thm:main}
  Let $R$ be a ring equipped with a non-archimedean valuation (see Definition \ref{def:valuation})
  $\nu : R \rightarrow S$, where $S$ is an idempotent semiring (such as $\T$), and let $X$ be a
  toric scheme over $\Fun$.  There is a tropicalization map
  \[
\Trop_X^\nu\co \{\text{closed subschemes of } X_R\} \to \{\text{closed subschemes of } X_S\}
\]
such that
\begin{enumerate}
\item it is functorial in $X$ with respect to torus-equivariant morphisms, and
\item when $S=\T$ the composition with $\Hom_{\Sch/\T}(\spec\T, -)$ recovers the set-theoretic
  tropicalization of Kajiwara-Payne.
\end{enumerate}
\end{thmA}

If $Z \subset X_R$ is irreducible of dimension $d$
and not contained in the toric boundary, then the restriction of the set-theoretic tropicalization $\trop(Z)$ to the
tropical torus admits the structure of a polyhedral complex of pure dimension $d$ and there are natural
number multiplicities associated to the facets such that the well-known balancing condition is
satisfied (see, e.g., \cite[\S2]{tropdiscrim}).  This balanced polyhedral complex is often thought
of as the tropical analogue of the algebraic cycle $[Z]$.  We show in Corollary
\ref{cor:multiplicities} that, when $Z$ is a hypersurface in projective space, the scheme
$\Trop_X^{\nu}(Z)$ determines these facet multiplicities.\footnote{In an earlier preprint version of this
  paper we conjectured this to be true in higher codimension, and Maclagan and Rinc\'on subsequently proved this \cite{Maclagan-Rincon}.}

\begin{thmA}\label{thm:HF}
  Let $\nu : k \rightarrow S$ be a valued field with $S$ a totally ordered idempotent semifield.
  Given a closed subscheme $Z\subset \PP^n_k$, the tropicalization $\Trop_{\PP^n}^\nu(Z) \subset
  \PP^n_S$ has a well-defined Hilbert polynomial, and it coincides with that of $Z$.
\end{thmA}

This suggests that the process of sending a variety to its tropicalization behaves like a flat
degeneration.

We briefly explain the idea behind the construction of this scheme-theoretic tropicalization.  Due
to the nature of $(\mathrm{max},+)$-algebra, the graph of a tropical polynomial $f$ is piecewise
linear; the regions of linearity are where a single monomial in $f$ strictly dominates and the
``bend locus,'' where the function is nonlinear, is the set of points where the maximum is attained
by at least two monomials simultaneously.  The bend locus (often called a tropical hypersurface or
locus of tropical vanishing) is the tropical analogue of the zero locus of a polynomial over a ring.
We enrich the bend locus of $f$ with a scheme structure by realizing it as the solution set to a
natural system of tropical algebraic equations, the \emph{bend relations} of $f$
(\S\ref{section:affinebendrels}).  These equations are given by equating $f$ with each polynomial
obtained from $f$ by deleting a single monomial.  By the Fundamental Theorem of Tropical Geometry
\cite[Theorem 3.2.4]{MS} (Kapranov's Theorem in the case of a hypersurface), set-theoretic
tropicalization can be recast as intersecting the bend loci of the coefficient-wise valuations of
all polynomials in the ideal defining an affine variety.  Our scheme-theoretic tropicalization is
defined by replacing this set-theoretic intersection with the intersection of bend loci as schemes.
This yields a solution to the implicitization problem for the coordinate-wise valuation map.

An alternative description of scheme-theoretic tropicalization is as follows.  An ideal $I$ defining
an affine subscheme is, in particular, a linear subspace of the coordinate ring
$k[x_1, \ldots, x_n]$ of the ambient space.  The set-theoretic tropicalization of $I$ is then an
(infinite-dimensional) tropical linear space $\trop(I)$ inside $\T[x_1, \ldots, x_n]$.  The
$\T$-module quotient of $\T[x_1, \ldots, x_n]$ by the bend relations of each $f\in \trop(I)$ turns
out to coincide with the $\T$-algebra quotient (Lemma \ref{lem:trop-lintrop}) and thus yields a
closed subscheme of tropical affine space.  For a homogeneous ideal defining a projective
subscheme, the bend relations are compatible with the grading, and thus the tropicalization can be
computed degree by degree in terms of $\T$-module theory and valuated matroids.  From this and the
fact that set-theoretic tropicalization preserves dimension, the Hilbert polynomial result follows.

The essential data in writing down the bend relations, and thus the scheme-theoretic
tropicalization, is a basis of monomials in the coordinate algebra of the ambient space.  Toric
varieties are a natural class of varieties where there is a well-behaved class of monomials in each
coordinate patch and this allows for a global extension of these affine constructions.  We use the
language of schemes over $\Fun$ (in the naive sense that is common to the theories of Durov,
Lorscheid, and To\"en-Vaqui\'e) as a convenient way to keep track of monomials and to provide a
slight generalization of the ambient toric varieties in which tropicalization usually takes place.
These $\Fun$ schemes are locally modelled on monoids; we no longer require the monoids to be toric,
though integrality (\S\ref{sec:monoids}) is still necessary.

One can ask how the tropicalization of $Z\subset X_R$ depends on the valuation $\nu: R \to \T$.  We
show that the valuations on $R$ form an algebraic moduli space and the corresponding
tropicalizations form an algebraic family over this space.

\begin{thmA}\label{thm:ValFamily} Let $R$ be a ring, $X$ a locally integral scheme over $\F_1$, and $Z \subset X_R$ a closed subscheme.
\begin{enumerate}
\item The moduli space $\val(R)$ of valuations on $R$ is represented in affine idempotent semiring
  schemes, and there is a universal valuation $\nu_{\mathrm{univ}} : R \rightarrow
  \Gamma(\val(R),\mathcal{O}_{\val(R)})$ through which all others factor uniquely.  In particular,
  $\val(R)(\T)$ is the set of non-archimedean valuations on $R$.
\item The fiber of the algebraic family $\Trop_X^{\nu_\mathrm{univ}}(Z) \rightarrow \val(R)$ over
  each $\T$-point $\nu : R \rightarrow \T$ is the tropicalization $\Trop_X^\nu(Z) \subset X_\T$.  If
  $X = \PP^n$ and $R$ is a field then the Hilbert polynomials of the fibers exist and are all equal.
\end{enumerate}
\end{thmA}

\begin{remark}
  If $k$ is a non-archimedean field and $R$ is a $k$-algebra, then the set of valuations on $R$
  extending the valuation on $k$ is the Berkovich analytification of $\spec R$ relative to $k$
  \cite{Berkovich}.  The $\T$-points of our moduli space $\val(R)$ can thus be viewed as an absolute
  version of the analytification of $\spec R$.  This is explored further in \cite{GG2}.
\end{remark}

\subsection{Organization of the paper}
We begin in \S\ref{sec:preliminaries} by recalling some standard material on monoids and semirings
and then giving our slightly generalized definition of valuation.  In \S
\ref{sec:scheme-theory} we discuss the construction of $\Fun$-schemes and semiring schemes, and in
\S \ref{sec:toric} we review some constructions in toric schemes within this setting.  The core of
the paper is \S\ref{sec:hypersurfaces}, where we define bend loci as schemes, and
\S\ref{sec:tropicalization}, where we use this to define and study scheme-theoretic tropicalization.
In \S\ref{sec:numerical-invariants} we study the tropical Hilbert function and the multiplicities on
the facets of a tropical hypersurface and in \S\ref{sec:hypbases} we investigate tropical bases.

\subsection{Acknowledgements}
The first author was supported by EPSRC grant EP/I005908/2, and the second author was supported by
NSF grant DMS-1204440.  We thank Dan Abramovich, Alex Fink, Mark Kambites, Eric Katz,
Oliver Lorscheid, Andrew MacPherson\footnote{MacPherson has been developing related ideas in his
  thesis \cite{MacPherson} and has independently discovered the equations for scheme-theoretic
  tropicalization that we propose here.}, Steven Sam, Bernd Sturmfels, and Sam Payne for useful
conversations and providing valuable feedback on early drafts.  Finally, we are indebted to Diane
Maclagan and Felipe Rinc\'on for reading a draft with great care and discussing these ideas at
length; these discussions helped shape this project.\footnote{These discussions also led Maclagan
  and Rinc\'on in \cite{Maclagan-Rincon} to further develop some of the ideas put forward in this
  paper.}

\section{Algebraic preliminaries: monoids, semirings, and valuations}\label{sec:preliminaries}

\emph{Throughout this paper all monoids, semirings, and rings will be assumed commutative and
  unital.  All monoids are viewed multiplicatively unless otherwise specified.}

\subsection{Monoids and $\Fun$ algebra}\label{sec:monoids}

In this paper we shall work with a naive version of algebra over the so-called ``field with one
element'', $\Fun$, which is entirely described in terms of monoids.  More sophisticated notions of
$\Fun$ algebra exist, such as Durov's commutative algebraic monads \cite{Durov}, but the naive
version recalled here is the one that appears most appropriate for tropical geometry and it provides
a convenient language for working with monoids and (semi)rings in parallel.  This naive $\Fun$
theory (or a slight variation on it) and its algebraic geometry have been studied by many authors,
including \cite{Connes-Consani, Deitmar, Toen-Vaquie, Weibel}.

Rather than defining an object $\Fun$, one starts by defining the category of modules, $\Fun$-Mod,
to be the category of pointed sets.  The basepoint of an $\Fun$-module $M$ is denoted $0_M$ and is
called the zero element of $M$.  This category has a closed symmetric monoidal tensor product
$\otimes$ given by the smash product of pointed sets (take the cartesian product and then collapse
the subset $M\times \{0_N\} \cup \{0_M\} \times N$ to the basepoint).  The two-point set $\{0,1\}$
is a unit for this tensor product.

An \emph{$\Fun$-algebra} is an $\Fun$-module $A$ equipped with a commutative and unital product map
$A\otimes A\to A$ (i.e., it is a commutative monoid in $\Fun$-Mod).  Concretely, an $\Fun$-algebra
is a commutative and unital monoid with a (necessarily unique) element $0_A$ such that $0_A \cdot x
= 0_A$ for all $x$; thus $\Fun$-algebras, as defined here, are sometimes called
\emph{monoids-with-zero}.  The two-point set $\{0,1\}$ admits a multiplication making it an
$\Fun$-algebra, and it is clearly an initial object, so we can denote it by $\Fun$ and speak of
$\Fun$-algebras without ambiguity.

\begin{example}
  The $\Fun$ polynomial algebra $\Fun[x_1, \ldots, x_n]$ is the free abelian monoid-with-zero on $n$
  generators.  The Laurent polynomial algebra $\Fun[x_1^{\pm 1}, \ldots, x_n^{\pm 1}]$ is the
  free abelian group on $n$ generators, together with a disjoint basepoint.  Written additively, these are $\mathbb{N}^n\cup\{-\infty\}$ and $\mathbb{Z}^n\cup\{-\infty\}$, respectively.
\end{example}

\begin{definition}\label{def:intmon}
An $\Fun$-algebra $A$ is \emph{integral} if the subset $A \smallsetminus \{0_A\}$ is
multiplicatively closed (no zero divisors) and the natural map from $A \smallsetminus \{0_A\}$ to
its group completion is injective. 
\end{definition}

An \emph{$A$-module} $M$ is an $\Fun$-module equipped with an
associative and unital action of $A$ given by a map $A\otimes M \to M$.  Concretely, this is a
pointed set with an action of the monoid $A$ such that $0_A$ sends everything to $0_M$.  An
\emph{$A$-algebra} is an $\Fun$-algebra morphism $A\to B$

\subsection{Semirings}\label{section:semirings}

Commutative monoids (viewed additively) admit a tensor product $\otimes$ generalizing that of
abelian groups.    If $M$ and $N$ are commutative monoids then $M\otimes N$ is the quotient of the free
commutative monoid generated by $M\times N$ by the relations 
\begin{enumerate}
\item $(x_1 + x_2,y) \sim (x_1,y) + (x_2,y)$, 
\item $(x,y_1 + y_2)\sim (x,y_1) + (x,y_2)$,
\item $(0,y) \sim 0$ and $(x,0) \sim 0 $,
\end{enumerate}
for all $x,x_1,x_2 \in M$ and $y,y_1,y_2 \in N$.  It satisfies the analogue of the usual universal
property: additive maps $M\otimes N \to L$ are in bijection with bilinear maps $M\times N \to L$.

Just as a ring can be defined as a monoid in the category of abelian groups, a \emph{semiring} is a
monoid in the monoidal category of commutative monoids---that is, an object satisfying all the
axioms of a ring except for the existence of additive inverses.  For a semiring $S$, an
\emph{$S$-module} is a commutative monoid $M$ equipped with an associative action
$S\otimes M \to M$.  An \emph{$S$-algebra} is a morphism of semirings $S\to T$.  Polynomial algebras
$S[x_1, \ldots, x_n]$, and Laurent polynomial algebras, are defined as they are for rings.  The
category of semirings has an initial object, $\N$, so the category of semirings is equivalent to the
category of $\N$-algebras.  A semiring is a \emph{semifield} if every nonzero element admits a
multiplicative inverse.

A semiring $S$ is \emph{idempotent} if $a + a = a$ for all $a\in S$.  In this case (and more generally, for an idempotent commutative monoid) there is a canonical partial order defined by
\[
a \leq b  \text{ if } a+b=b.
\]
The least upper bound of any finite set $\{a_i\}$ of elements exists and is given by the sum $\sum a_i$.  If
the partial order is actually a total order then $\sum a_i$ is equal to the maximum of the $a_i$.

From the perspective of tropical geometry, the central example of an idempotent semiring is the semifield of
\emph{tropical numbers}, $\T$.  As a set,
\[
\T \defeq \R \cup \{-\infty\}.
\] 
The addition operation is defined by the maximum: $a + b = \max\{a,b\}$ if both $a$ and $b$ are
finite.  Multiplication $a\cdot b$ in $\T$ is defined as the usual addition of real numbers $a + b$
if both are finite.  The additive and multiplicative units are $ 0_\T = -\infty$ and $1_\T = 0$,
respectively, and this defines the extension of addition and multiplication to $-\infty$.  

This is a special case of a general construction: given a commutative monoid $(\Gamma,+)$ equipped
with a translation-invariant total order, the set $\Gamma \cup \{-\infty\}$ equipped with the
operations $(\max,+)$ forms an idempotent semiring, and if $\Gamma$ is a group then this yields a
semifield.  These semifields are the targets of Krull valuations.  The tropical numbers $\T$ are the
rank 1 case, which results when $\Gamma$ is $(\R,+)$ with its canonical total order.  Taking $\Gamma=\R^n$
equipped with the lexicographic total order yields the idempotent semifield 
$\R^n_{lex} \cup \{-\infty\}$, which is the target of some higher rank Krull valuations and whose
tropical geometry has been studied in \cite{Banerjee}.

\begin{remark}
  Idempotent totally ordered semifields appear to play much of the role in idempotent algebra and
  geometry of fields in classical algebra and geometry.  
\end{remark}

The \emph{boolean semiring} is the subsemiring
\[
\mathbb{B} \defeq \{-\infty, 0\} \subset \T.
\]  
The boolean semiring is initial in the category of idempotent semrings and every
$\mathbb{B}$-algebra is idempotent, so $\mathbb{B}$-algebras are the same as idempotent semirings.

\subsection{Scalar extension and restriction}\label{sec:scalar-ext}

Given a (semi)ring $S$, there is an adjoint pair of functors
\[
\Fun\text{-Mod} \leftrightarrows S\text{-Mod};
\]
the right adjoint sends an $S$-module to its underlying set with the additive unit as the basepoint,
and the left adjoint, denoted $ - \otimes S$, sends a pointed set $M$ to the free $S$-module
generated by the non-basepoint elements of $M$.  If $M$ is an $\Fun$-algebra then $M\otimes S$ has
an induced $S$-algebra structure.  Note that $-\otimes S$ sends polynomial algebras over $\Fun$ to
polynomial algebras over $S$.

In this paper, $S$-modules equipped with an $\Fun$-descent datum (i.e., modules of the form
$M\otimes S$ for $M$ a specified $\Fun$-module) play a particularly important role.  For $f\in
M\otimes S$, the \emph{support} of $f$, denoted $\supp(f)$, is the subset of $M$ corresponding to
the terms appearing in $f$.

Given a semiring homomorphism $\varphi: S \to T$ one obtains an adjoint pair
\[
S\text{-Mod} \leftrightarrows T\text{-Mod}
\]
in the standard way.  As usual, the left adjoint is denoted
$-\otimes_S T$, and it sends $S$-algebras to $T$-algebras and
coincides with the pushout of $S$-algebras along $\varphi$.

\subsection{Ideals, congruences and quotients}

Let $A$ be either an $\Fun$-algebra or a semiring.  We can regard $A$ as an $A$-module and define an
\emph{ideal} in $A$ to be a submodule of $A$.  When $A$ is a ring this agrees with the usual
definition of an ideal.

Quotients of semirings generally cannot be described by ideals, since a quotient might identify
elements $f$ and $g$ without the existence of an element $f-g$ to identify with zero.  The same
issue arises when constructing quotients of modules over semirings.  For this reason, one must
instead work with congruences to describe quotients.  Here we we present some basic results
illustrating the behavior and use of congruences.  Further details on properties of congruences can
be found in the books \cite{Golan} and \cite{Hebisch-Weinert}.

\begin{definition}
  Let $S$ and $M$ be a semiring and $S$-module respectively.  A \emph{semiring congruence} on $S$ is
  an equivalence relation $J \subset S\times S$ that is a sub-semiring, and a \emph{module
    congruence} on $M$ is an $S$-submodule $J \subset M\times M$ that is an equivalence relation.
  If the type is clear from context, we refer to such an equivalence relation simply as a
  \emph{congruence}.
\end{definition}

Given a semiring congruence $J$ on $S$, we write $S/J$ for the set of equivalence classes.  Note
that, just as for rings, finite sums and products in the category of semirings agree, and moreover,
the pushout of the two projection maps $S\leftarrow J \to S$ has $S/J$ as its underlying set, and
\[
\begin{diagram}
\node{J} \arrow{e} \arrow{s} \node{S} \arrow{s} \\
\node{S} \arrow{e} \node{S/J}
\end{diagram}
\]
is both a pushout square and a pullback square in semirings; the analogous statement for
modules also holds.  From this one sees the following (see also \cite[Theorem 7.3]{Hebisch-Weinert}):
\begin{proposition}
  Let $J$ be an equivalence relation on a semiring $S$ (or module $M$ over a semiring).  The
  semiring (or module) structure descends to the set of equivalence classes $S/J$ (or $M/J$) if and
  only if $J$ is a semiring (or module) congruence.
\end{proposition}

\begin{definition}\label{def:congker}
  Given a morphism of semirings $\varphi : S \rightarrow R$, we define
  the \emph{kernel} congruence \[\ker\varphi := S\times_R S = \{(f,g)\in S\times
  S ~|~ \varphi(f)=\varphi(g)\}.\]
\end{definition}

Using congruences in place of ideals, the usual isomorphism theorems extend to semirings:
\begin{proposition}
\begin{enumerate}
\item Let $\varphi: S \to R$ be a homomorphism of semirings.  The image is a semiring, the
  kernel is a congruence, and $S/\ker \varphi \cong \im \varphi$.

\item Let $R$ be a semiring, $S\subset R$ a sub-semiring, $I$ a congruence on $R$, and let $S+I$
  denote the $I$-saturation of $S$ (the union of all $I$-equivalence classes that contain an element
  of $S$).  Then $S+I$ is a sub-semiring of $R$, $I$ restricts to a congruence $I'$ on $S+I$ and a
  congruence $I''$ on $S$, and there is an isomorphism $(S+I)/I' \cong S/I''$

\item For $J \subset I$ congruences on $S$, we have a congruence $I/J$ (the image of $I$ in
  $S/J \times S/J$) on $S/J$ with $(S/J) / (I/J) \cong S/I$.  This yields a bijection between
  congruences on $S/J$ and congruences on $S$ containing $J$.
\end{enumerate}
\end{proposition}
\begin{proof}
Part (1) is Theorem 7.5 in \cite{Hebisch-Weinert}, and part (3) follows directly from Theorem 7.12
there. For part (2), first observe that two elements $s_1, s_2 \in S$ are equivalent modulo $I''$ if and only if
they are equivalent modulo $I$ as elements in $R$.  From this it follows that $S/I''$ maps
isomorphically onto its image in $R/I$.  By construction, $(S+I)/I'$ has the same image in $R/I$ as
$S/I''$, and by the same reasoning it also maps isomorphically onto its image.
\end{proof}

Since the intersection of congruences is a congruence, for a collection $\{f_\alpha,g_\alpha \in
S\}_ {\alpha\in A}$ there is a unique smallest (or finest) congruence identifying $f_\alpha$ with
$g_\alpha$ for each $\alpha$; this is the congruence generated by the pairs $(f_\alpha,g_\alpha)$.  In
the case of a semiring congruence, we denote this by $\langle f_\alpha \sim g_\alpha
\rangle_{\alpha\in A}$. More generally, for any subset $J \subset S\times S$, we denote by $\langle
J\rangle$ the semiring congruence it generates.  If $\varphi: S \to R$ is a semiring (or module)
homomorphism and $J$ is a congruence on $S$, then $\varphi(J)$ need not be a congruence on $R$
because transitivity and reflexivity can fail; we denote by $\varphi_*J$ the congruence generated by
$\varphi(J)$.

\begin{lemma}\label{lem:CongGen}
  The semiring congruence $\langle f_\alpha \sim g_\alpha \rangle_{\alpha\in A}$ consists of the
  transitive closure of the sub-semiring of $S\times S$ generated by the elements
  $(f_\alpha,g_\alpha)$, $(g_\alpha, f_\alpha)$, and the diagonal $S\subset S\times S$.  The
  analogous statement for module congruences also holds.
\end{lemma}
\begin{proof}
  The sub-semiring generated clearly gives a binary relation that is symmetric and reflexive, so it
  suffices to check that if $R\subset S\times S$ is a sub-semiring, then the transitive closure $R'$
  is also a sub-semiring.  Let $x_1, \ldots, x_n$ and $y_1, \ldots, y_k$ be sequences of elements in
  $S$ such that each consecutive pair $(x_i,x_{i+1})$ and $(y_i, y_{i+1})$ is in $R$.  Thus $(x_1,
  x_n)$ and $(y_1, y_k)$ are in $R'$, and we must show that the product and sum of these are each in
  $R'$.  We may assume $k\leq n$, and by padding with $n-k$ copies of $(y_k,y_k)$, which is in $R'$ since $R'$ contains the diagonal, we can assume that
  $k=n$.  By adding or multiplying the two sequences term by term we obtain the result.
\end{proof}

An ideal $I$ in a semiring $S$ induces a congruence $\langle f \sim 0_S\rangle_{f\in I}$ on $S$.
 
\begin{proposition}\label{prop:idealcong}
  If $S$ is a ring then the above map from ideals to congruences is a bijection with inverse given
  by sending a congruence $\{f_i \sim g_i\}_{i\in \Lambda}$ to the ideal $\{f_i-g_i\}_{i\in \Lambda}$.
\end{proposition}
\begin{proof}
Congruences on $S$ are in bijection with quotients of $S$ as a semiring, and ideals in $S$ are of
course in bijection with quotients of $S$ as a ring, so it suffices to show that every semiring
quotient of $S$ is in fact a ring.  This is immediate since a homomorphism of semirings sends
additively invertible elements to additively invertible elements.
\end{proof}

However, if $S$ is a semiring that is not a ring then there can me multiple distinct ideals inducing
the same congruence.  For example, in the semiring $\N$ the maximal ideal $\N\smallsetminus \{1\}$
and the unit ideal $\N$ both induce the same congruence.  Indeed, the former congruence contains $(1,1) + (0,2) = (1,3)$ and $(3,0)$ so by transitivity it also contains $(1,0)$.  

More significantly, many congruences on semirings are not induced by an ideal as above.  For
instance, the tropical numbers $\T$ form a semifield and hence the only proper ideal is the trivial
one; however, there is a non-trivial congruence on $\T$ with quotient the Boolean semiring
$\mathbb{B}$.  Another example is provide by the diagonal $\A^1$ inside $\A^2 = \spec \T[x,y]$; this
congruence is generated by the relation $x\sim y$.

\subsection{Valuations}\label{sec:valuations}

The term ``non-archimedean valuation'' on a ring $R$ usually means a homomorphism of multiplicative monoids
$\nu: R\to \T$ satisfying $\nu(0_R) = -\infty$ and
the subadditivity condition $\nu(a+b) \leq \nu(a)+\nu(b)$ for all $a,b \in R$.  (Many
authors use the opposite sign convention, and some would call this a ``semi-valuation'' unless the
non-degeneracy condition $\nu^{-1}(-\infty) = 0$ holds.)  The subadditivity condition
appears semi-algebraic but, as observed in \cite{Manon}, it can be reformulated as an algebraic
condition:
\[
\nu(a+b)+\nu(a) + \nu(b) = \nu(a) + \nu(b).
\]
We use this observation in \S\ref{sec:moduli} when constructing the moduli space of valuations on a
ring.  

It is useful---for example, when studying families of tropical varieties---to allow a more
general codomain, so throughout this paper the term ``valuation'' shall refer to the following
generalization.  Note that, when passing from $\T$ to an arbitrary idempotent semiring, the
total order is replaced by a partial order (cf., \S\ref{section:semirings}).

\begin{definition}\label{def:valuation}
  A \emph{valuation} on a ring $R$ is an idempotent semiring $S$ (called the \emph{semiring of
    values}), and a map $\nu: R \to S$ satisfying, 
\begin{enumerate}
\item (unit) $\nu(0_R) = 0_S$, and $\nu(1_R)=1_S$,
\item (sign) $\nu(-1_R)=1_S$,
\item (multiplicativity) $\nu(ab) = \nu(a)\nu(b)$, 
\item (subadditivity) $\nu(a + b) + \nu(a) + \nu(b) = \nu(a) + \nu(b)$.
\end{enumerate}
for any $a,b\in R$.  A valuation $\nu$ is said to be \emph{non-degenerate} if $\nu(a)=0_S$ implies $a=0_R$. 
\end{definition}

For $S = \T$ this coincides with the usual notion of a non-archimedean valuation described above.
When $S$ is totally ordered the resulting valuations coincide with Krull valuations and considering
these leads to Huber's ``adic spaces'' approach to non-archimedean analytic geometry \cite{Huber}.  
Note that any valuation on a field is automatically non-degenerate.

\begin{remark}
  If $S$ is totally ordered then the sign condition $\nu(-1_R) = 1_S$ is implied by the multiplicativity
  and the unit condition $\nu(1_R)=1_S$ in Definition \ref{def:valuation}.  Indeed,
  $\nu(-1_R)^2=\nu(1_R)=1_S$ and the total ordering implies that square roots, when they exist, are
  unique.  For semirings that are not totally ordered this need not be the case.  We believe that
  the sign condition is fundamental to our generalization of valuations, as it is necessary for the
  functoriality of tropicalization (Proposition \ref{prop:trop-functoriality}) --- see Remark
  \ref{rem:sign-required}.  Note that the sign condition is also satisfied by the valuations
  considered in \cite{MacPherson}.
\end{remark}

\begin{lemma}\label{lem:val} Let $\nu: R \to S$ be a valuation and $a,b\in R$.
\begin{enumerate}
\item $\nu(a + b) + \nu(a) = \nu(a+b) + \nu(a) + \nu(b)$.
\end{enumerate}
Assume now that the partial order on $S$ is a total order.
\begin{itemize}
\item[(2)] If $\nu(a) \ne \nu(b)$ then $\nu(a+b) = \nu(a) + \nu(b)$.
\item[(3)] The image of $\nu$ is a subsemiring of $S$,
  and $R \twoheadrightarrow \im\nu$ is a valuation.
\end{itemize}
\end{lemma}

\begin{proof}
For part (1), let $x= a+b$ and $y=-a$.  The subadditivity equation, applied to $x$ and $y$,
becomes
\[
\nu(b) + \nu(a+b) + \nu(-a) = \nu(a+b) + \nu(-a),
\]
and using the sign condition $\nu(-1_R)=1_S$, this becomes the desired equation.

For (2), assume without loss of generality that $\nu(a) < \nu(b)$; we then need to show that $\nu(a+b) = \nu(b)$.  As above, we apply the subadditivity property with $x = a+b$ and $y=-a$, which here yields
\[
\nu(b) + \nu(a+b) = \nu(a) + \nu(a+b).
\]
Adding $\nu(a)$ to both sides then using idempotency, subadditivity, and the condition $\nu(a)+\nu(b) = \nu(b)$, we deduce that $\nu(b) = \nu(a)+\nu(a+b)$.  The result then follows from the total order hypothesis.

Part (3) follows immediately from part (2).
\end{proof}

A \emph{valued ring} is a triple $(R, S, \nu: R \to S)$ where $R$ is a ring and $\nu$ is a
valuation.  Valued rings form a category in which a morphism $\varphi: (R,S,\nu) \to (R',S', \nu')$
consists of a ring homomorphism $\varphi_1: R \to R'$ and a semiring homomorphism $\varphi_2: S \to
S'$ such that $\nu'\circ \varphi_1 = \varphi_2 \circ \nu$.  Note that the composition of a valuation
$\nu: R\to S$ with a semiring homomorphism $S\to S'$ is again a valuation.

As an illustration of the utility of considering the general class of valuations defined above, we
show that, for a fixed ring $R$, there exists a \emph{universal valuation}
$\nu^R_{univ}: R \to S^R_{univ}$ on $R$ from which any other valuation can be obtained by
composition with a unique semiring homomorphism.  This will be used to show that, as one varies the
valuation on $R$, the set of all tropicalizations of a fixed subscheme forms an algebraic family
over $\spec S^R_{univ}$ (Theorem \ref{thm:ValFamily} part (1)).  Recall that
$\B=\{0,-\infty\} \subset \T$ is the boolean semiring and consider the polynomial $\B$-algebra
$\B[x_a \:|\: a\in R]$ with one generator $x_a$ for each element $a\in R$.  The universal semiring
of values $S^R_{univ}$ is the quotient of $\B[x_a \:|\: a\in R]$ by the congruence generated by the
relations
\begin{enumerate}
\item $x_0 \sim 0_S$ and $x_1 \sim x_{-1} \sim 1_S$,
\item $x_a x_b \sim x_{ab}$ for any $a,b \in R$,
\item $x_{a+b} + x_a + x_b \sim x_a + x_b$ for any $a,b\in R$.
\end{enumerate}
The universal valuation $\nu^R_{univ}$ sends $a$ to $x_a$.

\begin{proposition}\label{val-univ-property}
  Given a valuation $\nu: R\to T$, there exists a unique homomorphism $\phi: S^R_{univ} \to T$ such
  that $\phi\circ \nu^R_{univ} = \nu$.  Hence valuations with semiring of values $T$ are in
  bijection with homomorphisms $S^R_{univ} \to T$
\end{proposition}
\begin{proof}
  The homomorphism $\phi$ is defined by sending each generator $x_a$ to $\nu(a)$.  Since the
  relations in $S^R_{univ}$ correspond exactly to the relations satisfied by a valuation, $\phi$ is
  well-defined.  Uniqueness is immediate.
\end{proof}

\section{$\Fun$-schemes and semiring schemes}\label{sec:scheme-theory}

\subsection{Construction of $\Fun$-schemes and semiring schemes}

The papers \cite{Toen-Vaquie}, \cite{Lorscheid-blueprints1}, and \cite{Durov} each construct
categories of schemes over semirings and some notion of $\Fun$.  For the purposes of the present
paper we do not require the full generality of their constructions, so we present below a
streamlined construction that follows the classical construction of schemes and yields a category
that admits a full embedding into each of their categories.

\begin{remark}
  Over a semiring, the category of schemes described here is equivalent to that of To\"en-Vaqui\'e,
  and it is a full subcategory of Durov's generalized schemes. There is a functor from it to
  Lorscheid's category of blue schemes, and this functor is faithful but not full. See \cite{F1land}
  for a comparison of these three threories over each of their notions of $\Fun$.
\end{remark}

The construction of schemes modelled on $\Fun$-algebras or semirings proceeds exactly as in the
classical setting of rings.  Let $A$ be a $Q$-algebra, where $Q$ is either a semiring or an
$\Fun$-algebra.  A proper ideal in $A$ is \emph{prime} if its complement is closed under
multiplication.  Given a prime ideal $\mathfrak{p} \subset A$, one can form the localization
$A_{\mathfrak{p}}$ via equivalence classes of fractions in the usual way.  As a space, the prime
spectrum $|\spec A|$ is the set of prime ideals in $A$ equipped with the Zariski topology in which
the open sets are the collections of primes not containing a given ideal (a basis is given by sets
of the form $D(f) = \{\mathfrak{p} \:\: | \:\: f\notin \mathfrak{p}\}$ for $f\in A$).  Any
$A$-module (or algebra) $M$ determines a sheaf $\widetilde{M}$ of $Q$-modules (or algebras) that
sends a principal open set $D(f)$ to the localization $M_f = A_f \otimes M$ in which $f$ is
inverted.  In particular, $A$ itself gives a sheaf of $Q$-algebras, and this is the structure sheaf
$\mathscr{O}_A$.

An \emph{affine scheme} (over $Q$) is a pair $(X,\mathscr{O})$ consisting of a topological space $X$
and a sheaf of $Q$-algebras that is isomorphic to a pair of the form $(|\spec A|, \mathscr{O}_A)$.
A general $Q$-scheme is a pair that is locally affine.  A morphism of schemes is a morphism of pairs
that is given in suitable affine patches by a homomorphism of $Q$-algebras.  As explained in
\cite[6.5.2]{Durov}, for rings this coincides with the usual construction in terms of locally ringed
spaces.  The category of affine $Q$-schemes is equivalent to the opposite of the category of
$Q$-algebras.

\begin{proposition}\label{prop:over-cat-equiv}
  Given a $Q$-algebra $A$, the category of $A$-schemes is canonically equivalent to the category of
  $Q$-schemes over $\spec A$.
\end{proposition}

A condition that will be fundamental later in our tropicalization construction is the following.  An $\Fun$-scheme is \emph{locally integral} if it admits an open affine cover by the spectra of integral $\Fun$-algebras (recall Definition \ref{def:intmon}).

\begin{proposition}\label{prop:locint}
  Any locally integral $\Fun$-scheme has a topological basis given by the spectra of integral $\Fun$-algebras.
\end{proposition}

\begin{proof}
This follows from the fact that any localization of an integral $\Fun$-algebra is integral.
\end{proof}

\begin{remark}
  A related construction is that of a Kato fan, which first appeared in Kato's seminal work on log
  geometry \cite{Kato} and was later explored in relation to tropicalization by Ulirsch
  \cite{Ulirsch-trop}.  An affine Kato fan consists of the prime spectrum of a monoid $M$, defined
  as above, but equipped with the structure sheaf associating to each basic open subset $D(f)$ the
  quotient of $M_f$ by its subgroup of units.  Thus in the case of a monoid-with-zero, an affine
  Kato fan is homeomorphic to an affine $\Fun$-scheme.  An arbitrary Kato fan is obtained by gluing
  affine Kato fans.  The advantage of these modified structure sheaves in Kato's construction is
  that they allow more flexible gluing.  For instance, associated to a toroidal embedding is a Kato
  fan homeomorphic to the set of generic points of the toroidal strata together with their
  specialization relations \cite[Remark 4.16(ii)]{Ulirsch-trop}, whereas $\Fun$-schemes more closely
  resemble toric varieties (cf., \S\ref{sec:toric}) since inverting all non-zero elements of an
  integral monoid yields the coordinate algebra of a (possibly infinite-dimensional) algebraic
  torus.  However, in Kato's framework a torus has no non-trivial regular functions; since the
  tropical geometry we generalize in this paper is based on subschemes of tori, this appears to be
  an insurmountable obstruction to using Kato fans for our purposes.  A further variant of these
  constructions is that of Artin fans, developed in \cite{ACMW,Abramovich-Wise} and extending ideas
  in \cite{Olsson}; these are algebraic stacks \'etale locally isomorphic to the stack quotient of a
  toric variety by its big torus.  Thus an Artin fan has the flexibility of Kato fans yet retains in
  its stack structure some of the crucial information that Kato fans discard.    Many of the
  constructions developed in this paper should extend from the context of $\Fun$-schemes and
  the Zariski topology to the context of Artin fans and an appropriate \'etale site.
\end{remark}

\subsection{Base change functors}

The scalar extension and restriction functors of \S\ref{sec:scalar-ext} admit globalizations that we
briefly describe here.

Using the fact that $\Fun$-Mod and $S$-Mod (for $S$ a semiring) are cocomplete, all fiber products
exist in the categories of $\Fun$-schemes and $S$-schemes and they are constructed in the usual way.
In particular, if $T$ is an $S$-algebra and $X$ is an $S$-scheme, then $X_T \defeq \spec T
\times_{\spec S} X$ exists and by Proposition \ref{prop:over-cat-equiv} it can be regarded as a
$T$-scheme.  Thus $\spec T \times_{\spec S} -$ defines a base change functor from $S$-schemes
to $T$-schemes, and this is the right adjoint of the forgetful functor (defined using Proposition
\ref{prop:over-cat-equiv}) that regards a $T$-scheme as an $S$-scheme.

For $R$ a ring or semiring, the scalar extension functor $- \otimes R$ clearly sends localizations
of $\Fun$-algebras to localizations of $R$-algebras, so it globalizes to give a base change functor
from $\Fun$-schemes to $R$-schemes.  Given an $\Fun$-scheme $X$, we write $X_R$ for the base change
of $X$ to $R$-schemes. This base change functor is right adjoint to the forgetful functor from
$R$-schemes to $\Fun$-schemes that globalizes the corresponding forgetful functor from $R$-Mod to
$\Fun$-Mod.  Given an $\Fun$-scheme $X$, the set of $R$-points of $X$ and of $X_R$ coincide and we denote this set by $X(R)$.

\subsection{Closed subschemes}\label{sec:closed}

At a formal level, the classical theory of schemes and the extended theory of semiring schemes are
nearly identical when considering open subschemes and gluing.  However, novel features appear when
considering closed subschemes; this is essentially because the bijection between ideals and
congruences (Proposition \ref{prop:idealcong}) fails in general for semirings. 

Over a ring, a closed immersion is a morphism $\Phi : Y \to X$ such that $\Phi(Y)$ is topologically a closed subspace of $X$, the induced map $Y \to \Phi(Y)$ is a homeomorphism, and the sheaf map $\Phi^\sharp: \mathscr{O}_X \to \Phi_*\mathscr{O}_Y$ is surjective.  These conditions on $\Phi$ are equivalent to requiring that $\Phi$ be an affine morphism with $\Phi^\sharp$ surjective.  A closed subscheme is then an equivalence class of closed immersions, where $\Phi : Y \to X$ and $\Phi' : Y' \to X$ are equivalent if there is an isomorphism $Y \cong Y'$ commuting with these morphisms.  There is a bijection between closed subschemes of $X$ and quasi-coherent ideal sheaves on $X$.  

Over a semiring, the equivalence between the above two characterizations of a closed immersion
breaks down; see Remark \ref{rem:closedpts} below.  The prevailing attitude (e.g., in \cite{Durov}),
and the choice that we follow here, is to adopt the second perspective: for a scheme $X$ over a
semiring, a \emph{closed immersion} is an affine morphism $\Phi: Y\to X$ such that
$\Phi^\sharp: \mathscr{O}_X \to \Phi_*\mathscr{O}_Y$ is surjective.  As before, a \emph{closed
  subscheme} of $X$ is an equivalence class of closed immersions into $X$.  Closed subschemes of
$\Fun$-schemes are defined in the same way.

A \emph{congruence sheaf} $\mathscr{J}$ on $X$ is a subsheaf of $\mathscr{O}_X\times\mathscr{O}_X$ such
that $\mathscr{J}(U)$ is a congruence on $\mathscr{O}_X(U)$ for each open $U \subset X$.  A congruence
sheaf is \emph{quasi-coherent} if it is quasi-coherent when regarded as a sub-$\mathscr{O}_X$-module
of $\mathscr{O}_X \times \mathscr{O}_X$.

\begin{proposition}
Let $S$ be a semiring.
\begin{enumerate}
\item Let $X = \spec A$ be an affine $S$-scheme.  Taking global sections induces a bijection between
quasi-coherent congruence sheaves on $\spec A$ and congruences on $A$.
\item For $X$ an arbitrary $S$-scheme, there is a bijection between closed subschemes of $X$ and
  quasi-coherent congruence sheaves on $X$.
\end{enumerate}
\end{proposition}
\begin{proof}
Part (1) follows directly from \cite[Corollary 5.5]{Hartshorne}, whose proof is unaffected by the
generalization from rings to semirings.

For (2), given a closed immersion $\Phi: Y \to X$, the congruence kernel $\ker\Phi^\sharp$ is a
quasi-coherent congruence sheaf.  Conversely, a quasi-coherent congruence sheaf on $X$
determines determines a closed subscheme of each affine open subscheme $U$, and the quasi-coherence
condition together with part (1) ensure that these glue together to form a well-defined closed
subscheme of $X$.
\end{proof}

\begin{remark}\label{rem:closedpts}
Curiously, morphisms that are scheme-theoretic closed immersions defined in this way are often not
closed embeddings at the level of topological spaces.  For instance, a point $\Phi: \spec \T
\rightarrow \A^n$ corresponding to a $\T$-algebra morphism $\varphi: \T[x_1,\ldots,x_n]
\twoheadrightarrow \T$ sending each $x_i$ to some finite value $\varphi(x_i)\in \R$ is a closed
immersion, but the image of this map is not Zariski closed---in fact, it is a dense point!  Indeed,
$\varphi^{-1}(-\infty) = \{-\infty\}$, which is contained in all primes, so every point of
$| \spec \T[x_1,\ldots,x_n] |$ is in the closure of the image of the point $|\spec \T|$.
\end{remark}

The congruence sheaf generated by a family of quasi-coherent congruence sheaves is again a quasi-coherent congruence sheaf, and this defines the intersection of the corresponding closed subschemes.

\begin{remark}
  One can view the prime spectrum and its Zariski topology as a technical scaffolding whose purpose
  is to define the functor of points, which is then regarded as the fundamental geometric object as
  in \cite{Toen-Vaquie}.
  For instance, as we see in the following example, the $\T$-points of a tropical variety more
  closely reflect familiar geometry than its prime spectrum.  Moreover, there is a natural topology on the $\T$-points of a $\T$-scheme such that closed subschemes, as defined above, induce closed subsets of the $\T$-points; see \cite[\S3.4]{GG2} where this notion is introduced and used to show that the Berkovich analytification of a scheme is homeomorphic to the $\T$-points of a certain tropicalization of the scheme.
\end{remark}

\subsection{Example: the affine tropical line}

The set of $\T$-points of the affine line $\A^1 = \spec \T[x]$ is clearly $\T$ itself, but the
ideal-theoretic kernels of the corresponding homomorphisms $\T[x] \to \T$ are all trivial except for
the point $x \mapsto -\infty$ for which the ideal is maximal.  On the other hand, one can of course
distinguish all these points using the congruence-theoretic kernel, by the First Isomorphism
Theorem.

The semiring $\T[x]$ has a rather intricate structure;  however, it admits a quotient with the same set of
$\T$-points that behaves more like univariate polynomials over an algebraically closed field: 
\[
\overline{\mathbb{T}[x]} := \T[x]/\sim, \text{ where }f\sim g \text{ if } f(t) = g(t) \text{ for all } t\in \T.
\]
Polynomials in this quotient split uniquely into linear factors.  More specifically, if 
\[
b_t := 0+t^{-1}x \in \overline{\mathbb{T}[x]} \text{ for } t\in \T^\times=\R \text{ and } b_{-\infty}
:= x \in \overline{\mathbb{T}[x]},
\]
then any element of $\overline{\mathbb{T}[x]}$ can be written uniquely as $c\prod b_{t_i}^{d_i}$ for
$c,t_i\in \T$.  Nonetheless, the prime spectrum of $\overline{\mathbb{T}[x]}$ is larger than one might
guess based on analogy with the case of algebraically closed fields.  For any subset $K \subset \T$
we define the ideal $I_K := (\{b_t~|~t\in K\}) \subset \overline{\mathbb{T}[x]}$.

\begin{proposition}
  If $K\subset\T$ is an interval (not necessarily closed or open) then $I_K\setminus\{-\infty\}$ is
  the set of functions that have a bend in $K$.  As a set, $|\spec \overline{\mathbb{T}[x]}| =
  \{I_K~|~K\subset \T\text{ is an interval }\}$.  The finitely generated primes correspond to closed
  intervals and the principal primes to points of $\T$.
\end{proposition}

\begin{proof}
  If $f\in \overline{\mathbb{T}[x]}$ has a bend at $t\in K\subset \T$ then $f\in I_{\{t\}} \subset
  I_K$.  Conversely, if $f\in I_K$ then $f = \sum_{i=1}^n g_i b_{t_i}$ for some $t_i\in K$ and
  $g_i\in \overline{\mathbb{T}[x]}$.  Each summand $g_ib_{t_i}$ has a bend at $t_i$, and the tropical
  sum of a function with a bend at $t_i$ and a function with a bend at $t_j$ must have a bend in the
  closed interval $[t_i,t_j]$.  Thus when $K$ is convex (i.e., an interval) we indeed have that the
  non-constant functions of $I_K$ are precisely the functions with a bend in $K$.

  From this it follows that if $K$ is an interval then $I_K$ is prime: if $f,g\in
  \overline{\mathbb{T}[x]}\setminus I_K$ then neither $f$ nor $g$ has a bend in $K$ so the same is
  true of $fg$, hence $fg\in \overline{\mathbb{T}[x]}\setminus I_K$.  Conversely, if
  $\mathfrak{p}\subset \overline{\mathbb{T}[x]}$ is prime then by the factorization property of
  $\overline{\mathbb{T}[x]}$, any element of $\mathfrak{p}$ must be divisible by $b_t$ for some
  $t\in \T$.  The identity
  \[
  t_1r^{-1}b_{t_1} + b_{t_2} = b_r \text{ for any }r \in [t_1,t_2] \subset \T
  \]
  then shows that $\mathfrak{p} = I_K$ where $K$ is the convex hull of all such $t$.  The statement
  about finitely generated primes and principal primes immediately follows.
\end{proof}

\section{Toric varieties and their tropical models}\label{sec:toric}

\subsection{Toric schemes over $\Fun$ and $\T$}

Let $N \cong \Z^n$ be a lattice with dual lattice $M$.  The datum of a rational polyhedral fan
$\Delta$ in $N_\R$ determines an $\Fun$-scheme as in the usual construction of toric varieties.  For
each cone $\sigma \in \Delta$, there is a corresponding monoid $M_\sigma = M \cap \sigma^\vee$.  If
$\tau\subset \sigma$ is a face then $M_\tau$ is a localization of $M_\sigma$.  Hence adjoining zeros
to these monoids, after writing them multiplicatively, and taking $\spec$ results in a collection of affine $\Fun$-schemes that glue
together according to the incidence relations of the fan $\Delta$ to give an $\Fun$-scheme
$X^\Delta$.  Base change to a ring $R$ yields the usual toric variety over $R$
associated with the fan $\Delta$.  The full subcategory of $\Fun$-schemes spanned by the objects of
the form $X^\Delta$ is equivalent to the category of toric varieties and torus-equivariant
morphisms.

Kajiwara \cite{Kajiwara} and Payne \cite{Payne} have each studied toric varieties over $\T$.  The
$\T$-points of the open torus stratum are canonically identified with the points of $N_\R \cong \mathbb{R}^n$, and
$X^\Delta(\T)$ is then the polyhedral partial compactification of $N_\R$ dual to the fan $\Delta$,
with a codimension $i$ stratum at infinity for each $i$-dimensional cone.  For example, $\PP^n(\T)$
is homeomorphic, and combinatorially equivalent, to an $n$-simplex.

\begin{remark}
Given a toric variety $X_k$, where $k$ is a valued field, some authors refer to the corresponding tropical
scheme $X_\T$ as the tropicalization of $X_k$.  
\end{remark}

Observe that the toric $\Fun$-schemes $X^\Delta$ described above are locally integral.  However, the class
of locally integral $\Fun$-schemes is larger; it allows objects that are disconnected, non-normal, and/or not of finite
type.  In the scheme-theoretic tropical geometry that we develop in this paper, the class of ambient
spaces in which tropicalization makes sense can naturally be enlarged from toric varieties to
locally integral $\Fun$-schemes.

\subsection{Cox's quotient construction}\label{sec:Cox}

We now explain how Cox's construction of toric varieties as quotients of affine space can be defined over $\F_1$.  Let $X= X^\Delta$ be as above and suppose the rays $\Delta(1)$ span $N_\R$, i.e., $X$ has no torus factors.  We define the \emph{Cox algebra} as the free $\F_1$-algebra on the set of rays: \[\Cox(X) := \F_1[x_\rho ~|~\rho\in\Delta(1)].\]

For any field $k$ the toric variety $X_k$ is split and the divisor class group is independent of the field $k$, so we can formally define $\Cl(X) := \Z^{\Delta(1)}/M$, where \[M \hookrightarrow \Z^{\Delta(1)},~ m \mapsto (m \cdot u_\rho)_{{\rho\in \Delta(1)}},\] and $u_\rho$ denotes the first lattice point on the ray $\rho \subset N_\R$.   

The Cox algebra has a grading by the divisor class group, via the composition 
\[
\Cox(X)\setminus\{0\} \cong \N^{\Delta(1)} \hookrightarrow \Z^{\Delta(1)} \twoheadrightarrow \Cl(X),
\]
where the above isomorphism is from a multiplicative to an additive monoid. The graded pieces are the eigenspaces for the action of the dual group 
\[
G := \Hom(\Cl(X),\Z) \subset \Hom(\Z^{\Delta(1)},\Z)
\]
on $\spec \Cox(X) \cong \mathbb{A}^{\Delta(1)}_{\F_1}$.

Each class $[D]\in \Cl(X)$ is represented by a torus-invariant Weil divisor $D$ and determines a coherent sheaf $\mathscr{O}_X(D)$ on $X$, the global
sections of which are naturally isomorphic to the $\F_1$-module of homogeneous elements in $\Cox(X)$
of multi-degree $[D]$.  If $X$ is complete then each graded piece is finite and the sections of this
$\F_1$-sheaf are naturally the lattice-points in a polytope.

The irrelevant ideal in $\Cox(X)$ is generated by the elements
$x_\sigma := \prod_{\rho \notin \sigma(1)} x_\rho$ for all cones $\sigma \in \Delta$.  This
determines an open subscheme
\[
U := \bigcup_{\sigma\in\Delta} \spec \Cox(X)[x^{-1}_\sigma] \subset \mathbb{A}_{\F_1}^{\Delta(1)}.
\]

\begin{proposition}\label{prop:F1coxquot}
With notation as above, $X$ is the categorical quotient $U/G$ in $\F_1$-schemes.
\end{proposition}

\begin{proof}
This is an immediate translation of \cite[Theorem 2.1]{Cox} and its proof to the setting of monoids.
We cover $U$ by $\F_1$-open affine $G$-invariant charts $U_\sigma := \spec\Cox(X)[x^{-1}_\sigma]$
and observe that Cox's argument carries over to show that \[\Cox(X)[x^{-1}_\sigma]^G =
\Cox[x^{-1}_\sigma]_0 \cong \sigma^\vee\cap M.\]  This implies that for this chart we have
the categorical quotient \[U_\sigma/ G = \spec \sigma^\vee \cap M,\] and following Cox's argument again we see that the way these affine quotients glue together to yield the categorical quotient $U/ G$ is identical to the way the affine charts corresponding to the cones in the fan $\Delta$ glue together to produce the toric variety $X$.
\end{proof}

\section{Bend loci}\label{sec:hypersurfaces}

In this section we define the \emph{bend locus scheme} of a tropical regular function; locally, a
bend locus is the tropical analogue of the zero locus of a regular function.  These bend loci will be the basic building blocks of scheme-theoretic
tropicalization.

Recall that over a ring $R$, a polynomial $f \in R[x_1, \ldots, x_n]$ determines a zero locus $V(f) \subset
\mathbb{A}^n_R$ as the set of points where $f$ vanishes, but $V(f)$ has the additional structure of
a scheme over $R$ given by $\spec R[x_1, \ldots, x_n]/(f)$.  There are various heuristic arguments
(e.g., \cite[\S3]{Sturmfels-first-steps}, \cite[\S3.1]{Mikhalkin-ICM}) that the correct analogue of
zero locus in the tropical setting is the locus of points where the piecewise linear graph of a
tropical polynomial $f$ is nonlinear---i.e., the locus where the graph ``bends''. This set appears
in the literature under various names: the tropical hypersurface, tropical vanishing locus, corner
locus, or bend locus of $f$.  We will refer to it as the \emph{set-theoretic bend locus}, and denote it $\setBend(f)$.

In embedding tropicalization into semiring scheme theory, the relevant question is then how to realize this set $\setBend(f)$ as the $\T$-points of a naturally associated semiring scheme, and to do so in a way that generalizes from affine space
to a larger class of $\F_1$-schemes and allows for coefficients in an arbitrary idempotent semiring
rather than just $\T$.  Associating a closed subscheme structure to the set-theoretic bend locus means realizing it as the set of solutions to a system of polynomial equations over
$\T$---more precisely, we must construct a congruence on the coordinate algebra of the ambient
affine scheme (and a quasi-coherent congruence sheaf in the non-affine case) such that the
$\T$-points of the quotient form the set-theoretic bend locus.  While the $\T$-points alone are not
enough to uniquely determine this congruence, the particular congruence we propose here appears
quite natural and allows for a robust theory of scheme-theoretic tropicalization to be developed.

Given an idempotent semiring $S$, an $\F_1$-algebra $M$ (which is the set of monomials) and
$f \in M\otimes S$, we construct a congruence $\langle \bend(f) \rangle$ which defines the bend
locus $\Bend(f)$ of $f$ as a closed subscheme of $\spec M\otimes S$.  (Note the capitalization: we use $\Bend$ for scheme-theoretic and $\setBend$ for set-theoretic bend loci, consistent with our use of $\Trop$ for scheme-theoretic and $\trop$ for set-theoretic tropicalization.)  The generators of this
congruence are called the \emph{bend relations} of $f$.  When the ambient space is a torus and
$S=\T$, the set of $\T$-points of $\Bend(f)$ equals $\setBend(f)$.  The scheme $\Bend(f)$
contains strictly more information than the set $\setBend(f)$.  It determines the
multiplicities (see \S\ref{sec:mults}), and while the set of $S$-points of $\Bend(f)$ does not in
general determine $f$ up to a scalar, the scheme structure does in some cases, such as when $S$ is a
semifield and $f$ is homogeneous (see Lemma \ref{lem:hyprecov} below).

\begin{remark}\label{rem:bend-locus-trop-difference}
  A word of caution: in the literature, the set-theoretic bend locus of a tropical polynomial is
  often called a ``tropical hypersurface,'' and the set-theoretic tropicalization of a hypersurface
  is an example of one.  Indeed, the set-theoretic tropicalization of the hypersurface $V(f)$ coincides with the
  set-theoretic bend locus of the tropical polynomial $\nu(f)$ obtained by coefficient-wise valuation: $\trop(V(f)) = \setBend(\nu(f))$.  However, this equality
  breaks at the level of schemes.  We shall define scheme-theoretic tropicalization, $\Trop$, below in
  \S\ref{sec:tropicalization} by taking the bend relations of the coefficient-wise valuations of
  \emph{all} elements in an ideal.  When enriched with this scheme structure, the tropicalization of
  a hypersurface is usually cut out by more relations than just the bend relations of a single
  tropical polynomial: $\Trop(V(f))$ is always a closed subscheme of $\Bend(\nu(f))$, but this containment can be strict.  See \S\ref{sec:vs} for examples where additional relations are needed and hence that $\Trop(V(f)) \subsetneq \Bend(\nu(f))$.
\end{remark}

\subsection{The bend relations}\label{section:affinebendrels}

Let $S$ be an idempotent semiring and $M$ an $\Fun$-module (or algebra).  Given $f\in M\otimes S$
and $m\in \supp(f)$, we write $f_{\widehat{m}}$ for the result of deleting the $m$-term from $f$.

\begin{definition}\label{def:setbend}
  The \emph{bend relations} of $f\in M\otimes S$ are the relations 
  \[
  \{f \sim f_{\widehat{m}}\}_{m\in \supp(f)}.
  \]
  We write $\bend(f)$ for the $S$-module congruence on $M\otimes S$ generated by the bend relations of
  $f$, and if $J\subset M\otimes S$ is an $S$-submodule then we write $\bend(J)$ for the $S$-module
  congruence generated by the bend relations of each $f\in J$; these congruences are called the
  \emph{bend congruences} of $f$ and $J$ respectively.  The set of $S$-module homomorphisms
  $M\otimes S/\bend(f) \to S$, or $S$-algebra homomorphisms $M\otimes S/\langle\bend(f)\rangle \to
  S$ when $M$ is an $\Fun$-algebra, is called the \emph{set-theoretic bend locus of $f$, denoted $\setBend(f)$}.
\end{definition}

\begin{example}
If $f = a_1 x_1 + a_2 x_2 + a_3 x_3 \in S[x_1,x_2,x_3]$ then the bend relations of $f$ are
\begin{align*}
a_1 x_1 + a_2 x_2 + a_3 x_3 & \sim a_2 x_2 + a_3 x_3 \\ & \sim a_1 x_1 + a_3 x_3 \\ & \sim a_1 x_1 + a_2 x_2.
\end{align*}
\end{example}

\begin{lemma}\label{units-lemma}
  Let $f \in M\otimes S$.
  \begin{enumerate}
\item If $\lambda\in S$ is a unit then $\bend(\lambda f)=\bend(f)$.
\item Suppose $\{f_i\}$ are elements in $M\otimes S$ and
  $N$ is the $S$-submodule they generate.  The module congruence $\bend(N)$ is equal to
  the module congruence generated by $\{\bend(f_i)\}$.
\end{enumerate}
\end{lemma}
\begin{proof}
The arguments are straightforward.
\end{proof}

The first item above is analogous to the fact that, classically, all nonzero scalar multiples of a polynomial define the same hypersurface.  The second item is used multiple times throughout this paper; it allows one to work explicitly with a set of generators for the relations defining the closed subschemes we shall be studying in the idempotent world.

The following result sheds light on the behavior described in Remark \ref{rem:bend-locus-trop-difference}, namely, the presence of relations in a scheme-theoretically tropicalized hypersurface beyond those coming from the defining polynomial.  Indeed, the second part of Lemma \ref{lem:trop-lintrop} below implies that these extra relations are due to the interplay between the valuation and the ring structure prior to tropicalizing, as opposed to something taking place entirely on the idempotent side of the story.  The first part of the lemma is also crucial to our theory and will be relied up heavily in \S\ref{sec:tropicalization} where tropicalization is studied; note in particular that, via the bend relations, ideals in an idempotent algebra play an important role when forming quotient algebras.

\begin{lemma}\label{lem:trop-lintrop}
  Suppose $M$ is an $\Fun$-algebra, $f\in M\otimes S$, and $J\subset M\otimes S$ is an ideal.
\begin{enumerate}
\item The module congruence $\bend(J)$ is in fact a semiring congruence.
\item If $M$ is integral (recall Definition \ref{def:intmon}) and $J=(f)$ is a principal ideal then the semiring congruence $\langle
  \bend(f)\rangle$ generated by $\bend(f)$ is equal to $\bend(J)$.
\end{enumerate}
\end{lemma}
\begin{proof}
  By Lemma \ref{lem:CongGen}, it suffices to show that $\bend(J)$ is closed under multiplication by
  bend relations $g\sim g_{\widehat{m}}$ for $g\in J$. 

  We first show that $\bend(J)$ is not just an $S$-module congruence but actually an $M\otimes
  S$-module congruence.  Multiplying a generating relation $g\sim g_{\widehat{m}}$ by a monomial
  $x\in M$ yields a relation of the form $xg\sim xg_{\widehat{m}},$ and we must show that this is in
  $\bend(J)$.  If multiplication by $x$ is injective then this is one of the generating relations of
  $\bend(J)$; if the map is not injective then let $b_1,\ldots, b_n$ be those monomials in $\supp(g)
  \smallsetminus \{m\}$ that are identified with $m$ after multiplication by $x$ and let $\lambda_i
  \in S$ be the coefficient of $b_i$ in $g$.  Then $xg \in J$, so $xg \sim (xg)_{\widehat{xm}}$ is
  a relation in $\bend(J)$, and adding $\sum_i \lambda_i b_i$ to both sides (using the idempotency
  of addition) yields the desired relation $xg \sim xg_{\widehat{m}}$.

  Now suppose $h \sim h'$ is an arbitrary relation in $\bend(J)$ and $g\sim g_{\widehat{m}}$ is a
  generating bend relation.  Then since $\bend(J)$ is an $M\otimes S$-submodule, the two relations
  \[
  h g  \sim h' g  \text{ and } h' g \sim h' g_{\widehat{m}}
  \]
  are both in $\bend(J)$, and hence, by transitivity, the relation $h g \sim
  h'g_{\widehat{m}} $ is as well, and this proves part (1).

  We now turn to part (2). By part (1) we have $\langle \bend(f) \rangle \subset \bend(J)$, so it suffices to prove the opposite containment.   By Lemma \ref{units-lemma} part (2), $\bend(J)$ is generated as a module
  congruence by the module congruences $\bend(xf)$ for $x\in M$. By the integrality hypothesis on
  $M$, multiplication by any monomial $x$ yields an injective map $M \to M$, and so $\bend(xf) =
  x\bend(f) \subset \langle \bend(f) \rangle$.
\end{proof}

\begin{remark}
  If $M$ is not an integral $\Fun$-algebra then the bend congruence of a principal ideal $J=(f)$ can
  be strictly larger than the semiring congruence $\langle \bend(f) \rangle$.  For example, if
  $M=\Fun[x,y]/\langle x^2 \sim xy\rangle$ and $f=x+y \in M\otimes S$ then $\bend(xf)$ contains the
  relation $x^2 \sim 0_S$, while the semiring congruence $\langle\bend(f)\rangle$ does not contain this relation.
\end{remark}

This next proposition states that the set-theoretic bend locus $\setBend(f)$ of a tropical polynomial (i.e., the
solution set to its bend relations $\bend(f)$, recall Definition \ref{def:setbend}) is the usual set-theoretic tropical hypersurface defined
by $f$.  However, since these set-theoretic bend loci are defined slightly more generally and in
order to emphasize the distinction that exists at the scheme-theoretic level between bend loci and
tropicalizations of hypersurfaces (see Remark \ref{rem:bend-locus-trop-difference}), we will adhere
to the `bend locus' terminology.

\begin{proposition}\label{prop:bend-locus}
  Let $S$ be a totally ordered idempotent semiring, $M$ be an $\Fun$-module and $f\in M\otimes S$.
\begin{enumerate}
\item An $S$-module homomorphism $p: M\otimes S \to S$ descends to the quotient by $\bend(f)$, hence yields a point of $\setBend(f)$, if and
  only if either the maximum of the terms of $p(f)$ is attained  at least twice or $p(f)=0_S$.  

\item If $M$ is an $\Fun$-algebra then the $S$-algebra homomorphisms $M\otimes S \to S$ descending to the quotient by $\langle \bend(f) \rangle$ are in bijection with the $S$-module homomorphisms $M\otimes S \to S$ descending to the quotient by $\bend(f)$; consequently, these two notions of $\setBend(f)$ coincide.

\item If $X = \spec \Fun[x_1^{\pm 1}, \ldots, x_n^{\pm 1}]$ is a torus and $f\in\T[x_1^{\pm 1}, \ldots, x_n^{\pm 1}]$, then $X(\T) = (\T^\times)^n = \R^n$ and $\setBend(f)$ is the subset of points at which the function $X(\T) \to \T$ defined by $f$ is nonlinear.
\end{enumerate}
\end{proposition}
\begin{proof}
  For part (1), first note if $p: M\otimes S \to S$ is an $S$-module homomorphism then $p(f)$ is a
  sum of terms given by evaluating $p$ on each monomial term of $f$.  Thus $p$ factors through the
  quotient by $\langle\bend(f)\rangle$ if and only if $p(f) = p(f_{\widehat{i}})$ for each $i$. This
  happens if and only if either $|\supp(f)| \ge 2$  and no single summand in $p(f)$ is strictly
  larger than all others, or if all summands are $0_S$.
 
  Part (2) is clear, so we now prove part (3).  A homomorphism $p: \T[x_1^{\pm 1}, \ldots, x_n^{\pm 1}] \to \T$ is
  determined by the $n$-tuple of tropical numbers $p(x_1),\ldots, p(x_n)\in \T^\times = \R$, so we
  identify $p$ with a point in $\R^n$.  This Euclidean space is divided into convex polyhedral
  chambers as follows.  For each term of $f$ there is a (possibly empty) chamber consisting of all
  $p$ for which that term dominates, with the interior consisting of points where this term strictly
  dominates.  Since $f$ is the tropical sum (Euclidean maximum) of its terms, the chamber interiors
  are where the graph of $f$ is linear and the walls are where the maximum is attained at least
  twice and hence the graph is nonlinear.
\end{proof}

\begin{remark}\label{rem:setbend-equal-trophyperplane}
  If $M$ is a finitely generated $\Fun$-module and $f\in M \otimes \T$, or if $M$ is a (Laurent)
  polynomial algebra over $\Fun$ and $f\in M\otimes \T$ is homogeneous of degree one, then by the
  preceding proposition $\setBend(f)$ is a tropical hyperplane.  Tropical hyperplanes were first
  defined in \cite{Speyer1} as the set-theoretic tropicalization of classical hyperplanes and later
  recast in terms of \emph{tropical Pl\"ucker vectors}, or equivalently, \emph{valuated matroids},
  in \cite{Speyer2}.  A more algebraic exposition of tropical hyperplanes, and tropical linear
  spaces more generally, closer to the spirit of this paper is Frenk's thesis \cite[Chapter
  4]{Frenk}.
\end{remark}

In general one cannot recover a tropical polynomial from its set-theoretic bend locus (consider,
e.g., $x^2+ax+0 \in \T[x]$ as $a\in \T$ varies).  In the case of homogeneous polynomials this is
manifest as the statement that the tropicalization of the Hilbert scheme of projective hypersurfaces
is not a parameter space for set-theoretically tropicalized hypersurfaces (see
\cite[\S6.1]{Tropical-Hilbert}).  The following result says in particular that when enriched with
its scheme structure, one can indeed recover, up to a scalar, a homogeneous tropical polynomial from
its bend locus.  This result is used in \S\ref{sec:mults} to show that the scheme structure of a tropicalized hypersurface determines the multiplicities associated to the facets of the corresponding balanced polyhedral complex.

\begin{lemma}\label{lem:hyprecov}
  Suppose $S$ is a semifield and $f\in M\otimes S$.
\begin{enumerate}
\item The module congruence $\bend(f)$ determines $f$ uniquely up to a scalar.\footnote{Maclagan and
    Rinc\'on \cite{Maclagan-Rincon} have subsequently observed that Proposition
    \ref{prop:bend-locus} combined with the duality theory of tropical linear spaces easily imples
    the following more general statement.  Suppose $M$ is a finitely generated $\Fun$-module and
    $L\subset M\otimes S$ is a tropical linear space.  Let $N$ denote the set-theoretic bend locus
    of $L$ (i.e., the intersection of all bend loci of elements $f\in L$); it is given by
    $N = \Hom_S(M\otimes S/\bend(L), S)$.  Then $N$ is a tropical linear space, and $L$ and $N$ are
    dual tropical linear spaces (in the sense of, e.g., \cite{Speyer2}): $L^\perp = N$ and
    $N^\perp = L$.  This allows one to recover $L$ from $\bend(L)$.}
\item If $M$ is an $\Fun$-algebra that admits a grading by an abelian group such that $M_0 = 0_M$
  and $f$ is homogeneous, then $\langle\bend(f)\rangle$ determines $f$ up to a scalar.
\end{enumerate}
\end{lemma}

\begin{remark}
  The hypotheses for (2) are satisfied by the Cox algebra of a toric scheme $X$ over $\F_1$ whose
  base change to a ring is proper.  We show below in \S\ref{sec:CoxHypersurf} that a homogeneous
  polynomial in $\Cox(X_\T)$ defines a closed subscheme of $X_\T$, generalizing the case of a
  homogeneous polynomial (in the usual sense) defining a closed subscheme of tropical projective
  space.
\end{remark}

\begin{proof}
  For (1), write $f = \sum_{i=1}^n a_im_i$ with $a_i\in S, m_i\in M$.  If $n=1$ then the
  result is obvious, otherwise consider the elements $\phi$ of the dual module $\Hom (M\otimes S,
  S)$ of the form $m_i \mapsto 0_S$ for all $i$ except two indices, say $j_1$ and
  $j_2$.  Such a homomorphism descends to the quotient by $\bend(f)$ if and only if
  $a_{j_1}\phi(m_{j_1}) = a_{j_2} \phi(m_{j_2})$.  In this way we recover the ratio of each pair of
  coefficients $a_{j_1},a_{j_2}$, and hence the vector of all coefficients $(a_1,\ldots,a_n)$ up to
  a scalar.  Item (2) follows from (1) since the hypotheses guarantee that $\langle\bend(f)\rangle_{\deg(f)} = \bend(f)$, where the latter is viewed as a congruence on the module $M_{\deg(f)}\otimes S$.
\end{proof}

\subsection{Unicity of the bend relations}\label{sec:unicity}

Let $S$ be a totally ordered idempotent semifield, $M$ a finitely generated $\Fun$-module (i.e., a
finite pointed set), and $f\in M\otimes S$.  In this section we show that the
congruence $\bend(f)$ of bend relations of $f$ is determined in a canonical way, via idempotent
linear algebra, by the set-theoretic tropical hyperplane of $f$, i.e., its set-theoretic bend locus
$\setBend(f)$.  To explain this result we must first develop some idempotent linear algebra.

Given an $S$-module $V$, we will write 
\[
V^\vee =\Hom_S(V,S)
\]
for the $S$-linear dual.  As usual, $V^\vee$ inherits an $S$-module structure via
$(s\cdot \alpha) (v) = \alpha(s\cdot v)$ for $s\in S$, $v\in V$ and $\alpha \in V^\vee$.

\begin{proposition}
There is a canonical isomorphism of $S$-modules $(M\otimes S)^\vee \cong M\otimes S$.
\end{proposition}
\begin{proof}
Modules of the form $M\otimes S$ are free with a canonical finite basis given by the
non-basepoint elements of the $\Fun$-module $M$.
\end{proof}

The dual of a quotient of $M\otimes S$ is canonically a submodule of $(M\otimes S)^\vee$.  Going in
the other direction is not quite so straightforward: the dual of an arbitrary submodule
$W\subset M\otimes S$ is canonically a quotient of $(M\otimes S)^\vee$ if and only if $W$ has the
property that every linear map $W \to S$ admits an extension to all of $M\otimes S$; conveniently,
by \cite[Lemma 3.2.1]{Frenk}, the assumption that $S$ is a totally ordered semiring implies that
this condition holds for any $W$.

Recall from Definition \ref{def:setbend} that the set-theoretic bend locus $\setBend(f)$ is defined
as the $S$-linear dual of the quotient $(M\otimes S)/\bend(f)$.  (As explained in Remark
\ref{rem:setbend-equal-trophyperplane}, the set-theoretic bend locus is essentially the tropical
hyperplane defined by $f$.)  In particular, in the setting here,
\[
\setBend(f) \subset (M\otimes S)^\vee
\]
is a submodule, and its $S$-linear dual, $\setBend(f)^\vee$, is thus a quotient of
$M\otimes S^{\vee\vee} \cong M\otimes S$.  The main result of this section identifies this quotient
precisely as the quotient by the bend relations of $f$.  Since one can easily recover a congruence
$J$ on a module $V$ from its associated quotient $V/J$ as the pullback of the diagram
$V \to V/J \leftarrow V$, this allows one to recover the congruence $\bend(f)$ from the set
$\setBend(f)$ together with its $S$-module structure.

\begin{theorem}\label{thm:double-dual}
  Let $S$ be a totally ordered idempotent semifield, $M$ a finitely generated $\Fun$-module, and
  $f\in M\otimes S$.  The canonical map from $M\otimes S /\bend(f)$ to its double dual is an
  isomorphism; in particular, there is a canonical isomorphism $\setBend(f)^\vee \cong M\otimes S /\bend(f)$.
\end{theorem}
\begin{proof}
  Consider the following commutative diagram of $S$-modules:
  \[
  \begin{diagram}
  \node{M\otimes S } \arrow{e,t}{\cong} \arrow{s,A} \node{(M\otimes S)^{\vee \vee}} \arrow{s} \\
  \node{M\otimes S / \bend(f)} \arrow{e} \node{M \otimes S / \bend(f)^{\vee\vee}.}
  \end{diagram}
  \]
  Since $(M\otimes S)/\bend(f)^\vee$ is a submodule of $(M\otimes S)^\vee$ and $S$ is totally ordered, the right vertical arrow
  is surjective, and hence the bottom horizontal arrow is also surjective.

  To show that the bottom arrow is also injective, we will show that if $g, g'\in M\otimes S$ are
  equal at each point $p$ in $\setBend(f)$, then they are equal in the
  quotient by $\bend(f)$.  We will examine their coefficients one at a time. 

  If $m\in M$ is a monomial in the complement of $\supp(f)$ then $f(m) = 0_S$ and so $m \in \setBend(f)$.  Since $g$ and $g'$ agree on $\setBend(f)$, the coefficients of $m$ in $g$ and $g'$ are identical.

  We next consider the coefficients of monomials in $\supp(f)$.  

  First suppose $\supp(f) = \{a\}$ is a single monomial. In this case $\bend(f)$ is spanned by the
  single relation $a \sim 0_S$.  Since the coefficients of $g$ and $g'$ are identical for
  all monomials $b\neq a$, $g$ and $g'$ are identified in the quotient by $\bend(f)$.

  Now suppose $|\supp(f)| \ge 2$.  For any $a\in M$, let $\chi_a \in (M\otimes S)^\vee$ denote the
  map sending $a$ to $1_S$ and all other basis elements to $0_S$.  For any pair of distinct elements
  $a,b \in \supp(f)$, consider the element $p_{ab} \in (M\otimes S)^\vee$ given by the formula
  \[
  p_{ab} = \left( \frac{1_S}{\chi_a(f)} \right) \chi_a + \left( \frac{1_S}{\chi_b(f)} \right)
\chi_b.
  \] 
  Idempotency of addition implies that $p_{ab}$ factors through the quotient by $\bend(f)$, i.e.,
  $p_{ab} \in \setBend(f)$, since for any $c\in \supp(f)$ we have
\[
p_{ab}(f) =  1_S + 1_S  = 1_S = p_{ab}(f_{\widehat{c}}). 
\]
  For each $a\in \supp(f)$, consider the element
  $g_a \defeq \chi_a(g)/ \chi_a(f) \in S$ and likewise for $g'$,  and let 
\[
m = \min_{a\in M} g_a, \quad \text{and}\quad m' = \min_{a\in M} g'_a.
\]
By hypothesis, $p_{ab}(g) = p_{ab}(g')$ for all $a,b\in \supp(f)$, which yields the set of equations
  \begin{equation}\label{eq:ggprime-rels}
  (\mathscr{R}_{ab}): g_a + g_b =  g'_a + g'_b.
  \end{equation}
  Modulo the congruence $\bend(f)$, we may assume the minima $m$ and $m'$ are each attained at least
  twice by the following argument.  If $m$ is attained only once by some $g_a$, and $g_b$ achieves
  the minimum of the remaining non-minimal terms, then 
  \[ 
  g = g + g_b f_{\widehat{a}}
  \]
  since for any $c\in M$, $\chi_c(g+g_b f_{\widehat{a}}) = \chi_c(g) + g_b\chi_c( f_{\widehat{a}})$,
  and if $c\notin \supp(f_{\widehat{a}}) = \supp(f)\smallsetminus a$ then this reduces to
  $\chi_c(g)$, while when $c \in \supp(f_{\widehat{a}})$ we have
  \[
  g_b\chi_c(f_{\widehat{a}}) \leq g_c \chi_c(f_{\widehat{a}}) = \chi_c(g),
  \]
  and so in all cases $\chi_c(g) = \chi_c(g+g_b f_{\widehat{a}})$.  Then
  \[
  g = g + g_b f_{\widehat{a}} \sim g + g_b f,
  \]
  and in the final expression the minimum is equal to $g_b$ and is attained at least twice (at the
  monomials $a$ and $b$), so we replace $g$ with this element, and likewise for $g'$.  Now, for $a$
  and $b$ such that $g_a=g_b=m$, the equation $(\mathscr{R}_{ab})$ implies that $m \geq m'$, and
  choosing $a$ and $b$ such that $g'_a=g'_b=m'$ we likewise see that $m \leq m'$.  Hence $m=m'$.

  Now let $a_1, \ldots, a_n$ be the elements of $\supp(f)$ ordered so that $g_{a_1} = g_{a_2} \leq
  \cdots \leq g_{a_n}$.  Since $m=m'$, the equation $(\mathscr{R}_{a_1 a_2})$ implies that $g_{a_1}
  = g_{a_2} = g'_{a_1} = g'_{a_2}$.  For any $k > 2$, $g_{a_k}$ and $g'_{a_k}$ are both greater than
  or equal to $m$, and so the equation $(\mathscr{R}_{a_1 a_k})$ implies that $g_{a_k} = g'_{a_k}$.
  Thus we have shown that $g$ and $g'$ are equal in $M\otimes S / \bend(f)$.
\end{proof}

\subsection{Functoriality of the bend relations}
Here we give two lemmas expressing how the bend relations behave with respect to maps induced by morphisms of $\Fun$-modules and $\Fun$-algebras.  These results are fundamental to the development that follows.  Recall that for an $\Fun$-module $M$ and a semiring $S$, we view the elements of $M$ as the monomials of the $S$-module $M\otimes S$.  An $\Fun$-module morphism $\varphi : M \to N$ is simply a map of pointed sets.  This induces an $S$-module homomorphism $\varphi\otimes \id_S : M\otimes S \to N\otimes S$ that we denote simply by $\varphi_S$.  If $\varphi$ is a morphism of $\Fun$-algebras then $\varphi_S$ is an $S$-algebra homomorphism.  In either case, $\varphi_S$ sends monomials to monomials, and in the latter case it is multiplicative on monomials.

Any $S$-algebra homomorphism $\psi: A \to B$ induces a map $\psi_*$ that sends ideals in $A$ to
ideals in $B$ by sending $J$ to the ideal generated by the image of $J$, and likewise for semiring congruences.

\begin{lemma}\label{functoriality-lemma1}
  Suppose $\varphi: M\to N$ is an $\Fun$-module morphism and $f\in M\otimes S$.
\begin{enumerate}
\item $(\varphi_S)_* \bend(f) \subset \bend{(\varphi_S(f))}$.
\item If $M$ and $N$ are $\Fun$-algebras and $\varphi$ is an $\Fun$-algebra morphism then
  $(\varphi_S)_* \langle\bend(f)\rangle \subset \langle\bend{(\varphi_S(f))}\rangle$, and if
  $J\subset M\otimes S$ is an ideal then $(\varphi_S)_* \bend(J) \subset \bend((\varphi_S)_*J)$.
\item The above inclusions are equalities if $\varphi$ is injective.
\end{enumerate}
\end{lemma}
\begin{proof}
  Since $(\varphi_S)_* \bend(f) $ is generated as an $S$-module congruence by the image of the
  generators of $\bend(f)$, it suffices to show that any relation of the form
  $\varphi_S(f) \sim \varphi_S(f_{\widehat{m}})$ for $m\in M$ is implied by the relation
  $\varphi_S(f) \sim \varphi_S(f)_{\widehat{\varphi(m)}}$ in $\bend(\varphi_S(f))$.  Let
  $g_0, \cdots, g_k$ be the terms of $f$ corresponding to the monomials in
  $\varphi^{-1}(\varphi(m)) \subset M$, with $g_0$ being the term of $f$ whose support is $m$.  The
  relation $\varphi_S(f)_{\widehat{\varphi(m)}} \sim \varphi_S(f)$ implies
\begin{align*}
\varphi_S(f_{\widehat{m}}) = \varphi_S(f)_{\widehat{\varphi(m)}} + \varphi_S( g_1 + \cdots + g_k) & \sim
\varphi_S(f) + \varphi_S( g_1 + \cdots + g_k) \\
& = \varphi_S(f + g_1 + \cdots + g_k) = \varphi_S(f),
\end{align*}
where the last equality follows from the idempotency of addition in $S$.  This proves (1), from
which part (2) follows immediately.
 When $\varphi$ is injective it is clear that $\varphi_S(f_{\widehat{m}}) =
 \varphi_S(f)_{\widehat{\varphi_S(m)}}$, and so (3) follows from (1) and (2).
\end{proof}

The next lemma asserts that formation of the bend congruences commutes with monomial localization.
\begin{lemma}\label{qc-localization-lemma}
  Let $M$ be an $\Fun$-algebra, $\varphi: M \to M[x^{-1}]$ a localization map, and $J\subset M\otimes S$ an ideal.  The induced localization $\varphi_S : M\otimes S \to M[x^{-1}]\otimes S \cong (M\otimes S)[x^{-1}]$ satisfies $(\varphi_S)_*\bend(J) = \bend((\varphi_S)_*J)$.
\end{lemma}
\begin{proof}
  By Lemma \ref{functoriality-lemma1}, there is a containment $(\varphi_S)_*\bend(J) \subset
  \bend((\varphi_S)_*J)$ which is an equality if $\varphi$ is injective (for example, if $M$ is integral).  So suppose that $\varphi$ is not injective.  The congruence $\bend((\varphi_S)_*J)$
  is generated by relations of the form
\[
\varphi_S(f) \sim \varphi_S(f)_{\widehat{m}}
\]
for $f\in J$ and $m \in \supp(\varphi_S(f))\subset M[x^{-1}]$.  Note that $\varphi(m) = \varphi(n)$ if and only if $x^k m = x^k n$ for some sufficiently large natural number
$k$.  Since $f\in J$ implies $x^kf \in J$, it now suffices to observe that $\varphi_S(f) =
x^{-k}\varphi_S(x^kf)$ and $\varphi_S(f)_{\widehat{m}} =  x^{-k}\varphi_S((x^kf)_{\widehat{x^k m}})$.
\end{proof}

\subsection{Construction of bend loci}

Let $S$ be an idempotent semiring and $X$ an $\Fun$-scheme.  Let $X_S$ be the base change of $X$ from
$\Fun$ to $S$.  We now construct bend loci as closed subschemes of $X_S$.  

\begin{lemma}\label{lem:qcohcong}
Given a quasi-coherent ideal sheaf $\mathscr{I} \subset \mathscr{O}_{X_S}$, the association \[\spec M\otimes S \mapsto \bend(\mathscr{I}(\spec M\otimes S))\] defined for affine open subschemes  $\spec M \subset X$ determines a quasi-coherent congruence sheaf $\bend(\mathscr{I})$ on $X_S$.
\end{lemma}

\begin{proof}
For $\spec M \subset X$, the ideal $\mathscr{I}(\spec M\otimes S) \subset M\otimes S$ determines a semiring congruence $\bend(\mathscr{I}(\spec M\otimes S))$ on $M\otimes S$ by Lemma \ref{lem:trop-lintrop}(1).  By Lemma \ref{qc-localization-lemma}, formation of this congruence commutes with restriction to a smaller
open affine subscheme.
\end{proof}

\begin{definition}\label{def:bendideal}
  The \emph{bend locus} of a quasi-coherent congruence sheaf $\mathscr{I}$ on $X_S$ is the closed subscheme $\Bend(\mathscr{I}) \subset X_S$ defined by the quasi-coherent congruence sheaf $\bend(\mathscr{I})$.  If $X=\spec M\otimes S$ is affine and $I\subset M\otimes S$ is an ideal, we denote the corresponding closed subscheme simply by $\Bend(I)$.
  \end{definition}
 
  If $\mathscr{L}$ is a line bundle on $X$ (i.e., a locally free sheaf of rank one) and
  $f \in \Gamma(X_S, \mathscr{L}_S)$ is a global section of the base change to $S$, then $f$
  determines a quasi-coherent ideal sheaf: for a local trivialization
  $\mathscr{L}|_U \cong \mathscr{O}_X|_U$ this is given by the principal ideal generated by the
  image of $f|_{U_S}$ in $\mathcal{O}_{X_S}(U_S)$.  We denote the corresponding closed subscheme by
  $\Bend(f) \subset X_S$.  Note that, by Lemma \ref{lem:trop-lintrop}(2) and Proposition
  \ref{prop:locint}, if $X$ is locally integral then the bend relations locally defined by $f$ and
  by the principal ideal generated by $f$ coincide.  Since we will be primarily concerned with
  situations in which $X$ is locally integral, we can and will use the notation $\Bend(f)$
  without ambiguity.
 
\begin{proposition}
  Let $\varphi: X \to Y$ be a morphism of $\Fun$-schemes, $\mathscr{L}$ a line bundle on $Y$, and
  $f \in \Gamma(Y_S, \mathscr{L}_S)$.  Then $\varphi_S: X_S \to Y_S$
  maps $\Bend(\varphi_S^*f)$ into $\Bend(f)$.
\end{proposition}
\begin{proof}
It suffices to check on affine patches, where the result follows from Lemma \ref{functoriality-lemma1}.
\end{proof}

\subsection{Tropical Proj and Cox}\label{sec:CoxHypersurf}

If $S$ is a semiring and $A$ is an $\mathbb{N}$-graded $S$-algebra, then we can form an $S$-scheme $\proj A$ in the usual way, either by topologizing the set of homogeneous prime ideals in $A$ and constructing a structure sheaf from them, or by gluing the affine patches obtained by dehomogenizing.  We say that a congruence $J$ on $A$ is \emph{homogeneous} if the grading on $A$ descends to the quotient by $J$.  In this case, $\proj A/J$ is naturally a closed subscheme of $\proj A$.

Suppose $M$ is an $\N$-graded monoid-with-zero.  Then $M\otimes S$ is an $\N$-graded $S$-algebra, and if $I \subset M \otimes S$ is a homogeneous ideal then the congruence $\bend(I)$ is homogeneous.  Indeed, the bend relations of a homogeneous polynomial are manifestly homogeneous, and $\bend(I)$ is generated by the bend relations of the homogeneous elements of $I$.  We denote the resulting closed subscheme $\proj M\otimes S/\bend(I)$ by $\Bend(I)$.  As usual, the context should make it clear whether $\Bend(I)$ refers to this projective subscheme or to the affine scheme $\spec M\otimes S/\bend(I)$.  Implicit here is the claim that using the bend relations of homogeneous elements in $I$ to define a scheme via the Proj construction is equivalent to gluing the affine subschemes obtained by taking the bend relations after dehomogenizing the ideal.  This is a special case of Proposition \ref{prop:cox-compat} below.  

More generally, let $X=X^\Delta$ be a toric variety over $S$ without torus factors and consider its
$\Cl(X)$-graded algebra $\Cox(X) = S[x_{\rho}~|~\rho\in\Delta(1)]$.  Identical to the case of toric schemes over $\Fun$ discussed in \S\ref{sec:Cox}, each divisor class $[D] \in \Cl(X)$ is represented by a torus-invariant Weil divisor $D$ and the homogeneous polynomials in $\Cox(X)$ of multi-degree $[D]$ are in bijection with global sections of a coherent sheaf $\mathscr{O}_X(D)$.  Just as in the usual setting of toric varieties, these coherent sheaves need not be line bundles, but there is an inclusion $\Pic(X) \subset \Cl(X)$ of line bundle isomorphism classes into divisor classes and via this inclusion we can identify the global sections of any line bundle on $X$ with the homogeneous elements of $\Cox(X)$ of the corresponding multi-degree.

Recall (Proposition \ref{prop:F1coxquot}) that $X = U/G$, where $U$ is the complement of the vanishing of the irrelevant ideal and $G=\Hom(\Cl(X),\mathbb{Z})$.  Suppose $I \subset \Cox(X)$ is a homogeneous ideal generated by global sections of line bundles on $X$.  Then every homogeneous $f\in I$ is the global section of a line bundle and we can consider the intersection of the bend loci $\Bend(f) \subset X$ defined by these global sections, or we can consider the affine subscheme $\Bend(I) \subset \spec \Cox(X) \cong \mathbb{A}_S^{\Delta(1)}$ and attempt to take its quotient by the group $G$ to get a closed subscheme of $X$.  The following result says that these two constructions coincide.

\begin{proposition}\label{prop:cox-compat}
Let $I \subset \Cox(X)$ be as above.   Then \[\bigcap_{f \in I\text{ homogeneous}}\Bend(f) = (\Bend(I)\cap U)/G\] as closed subschemes of $X$, where the latter is the categorical quotient in the category of $S$-schemes.
\end{proposition}

\begin{proof}
Since $I$ is generated as an $S$-module by its homogeneous elements, Lemmas \ref{units-lemma}(2) and \ref{lem:trop-lintrop}(2) reduce us to proving the result in the case of a single homogeneous polynomial: $I =(f)$.

  For each cone $\sigma\in\Delta$, let $x_{\widehat{\sigma}} := \prod_{\rho \notin
    \sigma(1)}x_\rho$.  The restriction of $\Bend(f)$ to the affine open $\spec
  \Cox(X)[x^{-1}_{\widehat{\sigma}}] \subset U$ is defined by $\iota_*\langle\bend(f)\rangle$, where $\iota :
  \Cox(X) \rightarrow \Cox(X)[x^{-1}_{\widehat{\sigma}}]$ is the localization map.  The subalgebra
  of $G$-invariants on this chart is the degree zero
  piece \[(\Cox(X)[x^{-1}_{\widehat{\sigma}}]/\iota_*\langle\bend(f)\rangle)_0 =
  \Cox(X)[x^{-1}_{\widehat{\sigma}}]_0/\iota_*\langle\bend(f)\rangle_0,\] so this defines the restriction of $(\Bend(f)
  \cap U)/G \subset U/G$ to the affine open \[X_{\sigma} := \spec
  (\Cox(X)[x^{-1}_{\widehat{\sigma}}]_0) \subset X = U/G.\] On the other hand, a trivialization on
  $X_\sigma$ of a line bundle $\mathscr{L}$ on $X$ for which $f$ is a section corresponds to a choice of unit
  $g\in \Cox(X)[x^{-1}_{\widehat{\sigma}}]$ with $\deg(g) = \deg(f)$.  Then the bend locus of $f \in \Gamma(X,\mathscr{L})$ is defined on this affine patch by $\langle\bend(\frac{f}{g})\rangle = \iota_*\langle\bend(f)\rangle_0$, exactly as above.
\end{proof}

\section{Scheme-theoretic tropicalization}\label{sec:tropicalization}

Let $\nu\co k \to \T$ be a nontrivially valued field.  The usual set-theoretic tropicalization is a map,
 \[
\trop\co \{\text{subvarieties of } k^n\} \to \{\text{subsets of } \T^n\}.
\]
Given $Z\subset k^n$, the set $\trop(Z)$ can be described either as the Euclidean closure of the image of $Z$ under coordinate-wise valuation, or as the intersection of
the set-theoretic bend loci of the coefficient-wise valuations of the polynomials in the defining
ideal $I_Z$ of $Z$: 
\begin{equation}\label{eqn:settrop}
\trop(Z) = \overline{\nu(Z)} = \bigcap_{f\in I_Z} \setBend(\nu(f)).
\end{equation}
If $Z$ is a linear subvariety, then $\trop(Z) \subset \T^n$ is a $\T$-submodule.

In this section we construct and study a scheme-theoretic refinement/generalization $\Trop$ of the
set-theoretic tropicalization map $\trop$.  The setup is as follows.  Let $X$ be a locally integral $\Fun$-scheme
and $\nu: R \to S$ a valued ring.  Our scheme-theoretic tropicalization is then a map of posets (with
respect to inclusion)
 \[
\Trop^\nu_X\co \{\text{closed subschemes of } X_R\} \to \{\text{closed subschemes of } X_S\}
\]
(we will often drop the superscript and subscript if they are clear from the context).

Locally on an affine patch $\spec A \subset X$ it sends $Z \subset \spec A \otimes R$ to the scheme-theoretic intersection of the bend loci of the
coefficient-wise valuations of all functions in the defining ideal $I_Z \subset A \otimes R$: 
\[\Trop(Z) = \bigcap_{f\in I_Z}\Bend(\nu(f)).\]
This is functorial in $X$,
compatible with the Cox construction, and when $S=\T$ the composition with $\Hom_{\Sch/\T}(\spec \T,
-)$ recovers the extended tropicalization functor of Kajiwara-Payne (which generalizes the above
set-theoretic tropicalization from affine space to toric varieties).  Moreover, these
scheme-theoretic tropicalizations form an algebraic family as the valuation varies.

\subsection{Construction of the tropicalization functor}\label{sec:tropconstruct}

We first construct $\Trop^\nu_X$ in the case when $X=\spec A$ is an affine $\Fun$-scheme and then
show that these affine pieces glue together to define it in general.  In the affine case the
scheme-theoretic tropicalization sends ideals in $A\otimes R$ to semiring congruences on $A\otimes
S$, and it then globalizes to send quasi-coherent ideal sheaves to quasi-coherent congruence sheaves.

Let $M$ be an $\Fun$-module and consider the following two operations.
\begin{enumerate}
\item Given an $R$-submodule $N\subset M\otimes R$, let $\trop(N)\subset M\otimes S$ denote the
  $S$-submodule generated by the image of $N$ under the coefficient-wise valuation map $\nu:
  M\otimes R \to M\otimes S$.  (Note that when $R$ is a field and $S=\T$ then this is just the usual
  set-theoretic tropicalization map applied to linear subspaces.) 

\item The bend congruence construction $\bend(-)$ sending $S$-submodules of $M\otimes S$ to $S$-module
  congruences on $M\otimes S$.
\end{enumerate}
The composition of these operations, $N \mapsto \bend\trop(N)$, sends $R$-submodules to $S$-module
congruences.  Now suppose $A$ is an $\Fun$-algebra and $I\subset A\otimes R$ is an ideal.  Regarding
$I$ simply as a module, we can apply the above two operations to form the $S$-module congruence
$\bend\trop(I)$ on $A\otimes S$.

\begin{proposition}\label{prop:trop-sends-ideals-to-ideals}
If $A$ is integral then the $S$-submodule $\trop(I)\subset A\otimes S$ is an ideal, and hence
$\bend\trop(I)$ is a semiring congruence.
\end{proposition}
\begin{proof}
  It suffices to check that $\trop(I)$ is closed under multiplication by monomials.  An arbitrary
  element of $\trop(I)$ is a linear combination of elements of the form $\nu(f)$ for $f\in I$.  If
  $x$ is a monomial and $f\in I$ then $x\nu(f) = \nu(xf) \in \trop(I)$ since multiplication by $x$
  is injective on monomials.  The second claim then follows from Lemma \ref{lem:trop-lintrop} part (1).
\end{proof}

\begin{remark}
The intergality hypothesis is necessary here, as the following example shows.  Let
$A=\Fun[x,y]/\langle x^2 \sim xy \sim y^2\rangle$ and let $f = x-y \in A\otimes R$.  Then $\trop$ sends the
principal ideal $(f)$ to the $S$-submodule  spanned by the single element $x+y$, and this is not an ideal. 
\end{remark}

\begin{definition}[Affine tropicalization]\label{def:affinetrop}
  If $A$ is an integral $\Fun$-algebra and $Z \subset \spec A\otimes R$ is the closed subscheme corresponding
  to an ideal $I$, then we define $\Trop(Z) \subset \spec A\otimes S$ to be the closed subscheme
  determined by the semiring congruence $\bend\trop(I)$, i.e., 
\[
\Trop(Z) := \Bend(\trop(I)).
\]
\end{definition}

\begin{proposition}\label{prop:intersection-of-bend-loci}
  The subscheme $\Trop(Z)$ is the scheme-theoretic intersection $\bigcap_{f\in I}\Bend(\nu(f)).$
\end{proposition}
\begin{proof}
The set $\{\nu(f)\}_{f\in I}$ generates $\trop(I)$ as an $S$-module by definition, and so by Lemma \ref{units-lemma}
part (2), the congruence $\bend\trop(I)$ is generated as an $S$-module congruence by
$\{\bend(\nu(f))\}_{f\in I}$.
\end{proof}

We now globalize the above picture.  Let $X$ be a locally integral $\Fun$-scheme and $\mathscr{I}$ a quasi-coherent
ideal sheaf on $X_R$.  If $U=\spec A$ is an affine patch then over $U_R$ the sheaf $\mathscr{I}$ is
given by an ideal $I\subset A\otimes R$.  By taking $U$ to be sufficiently small we can, by Proposition \ref{prop:locint}, assume that $A$ is an integral $\Fun$-algebra.  Then $\trop(I)$ is an ideal in $A\otimes S$, by Proposition
\ref{prop:trop-sends-ideals-to-ideals}.  Hence we have a quasi-coherent ideal sheaf
$\trop(\mathscr{I}|_{U_R})$ over the affine patch $U_S$.  The following lemma shows that these
locally defined  sheaves assemble to give a well-defined quasi-coherent ideal sheaf
$\trop(\mathscr{I})$ over $X_S$.

\begin{lemma}\label{lem:trop-localization-compatibility}
If $A$ is integral then the construction $I \mapsto \trop(I)$ commutes with monomial localizations.  I.e., given a
localization map $\varphi: A \to A[x^{-1}]$, one has $\trop((\varphi_R)_*I) = (\varphi_S)_* \trop(I)$. 
\end{lemma}
\begin{proof}
The ideal $\trop((\varphi_R)_*I)$ is spanned as an $S$-module by the elements of the form
$\nu(x^{-n}f)$ for $f\in I$.  Since $A$ is integral, multiplication by $x$ commutes with valuation
and therefore $\nu(x^{-n}f)=x^{-n}\nu(f)$, and elements of this form span $(\varphi_S)_* \trop(I)$.
\end{proof}

By Lemma \ref{qc-localization-lemma}, applying $\bend(-)$ locally to the quasi-coherent ideal sheaf
$\trop(\mathscr{I})$ yields a quasi-coherent congruence sheaf $\bend\trop(\mathscr{I})$.

\begin{definition}
  Let $X$ be a locally integral $\Fun$-scheme and $Z\subset X_R$ a closed subscheme defined by an ideal
  sheaf $\mathscr{I}_Z$.  The scheme-theoretic tropicalization of $Z$ is the subscheme
  $\Trop^\nu_X(Z)\subset X_S$ defined by the congruence sheaf $\bend\trop(\mathscr{I}_Z)$, i.e., 
\[
\Trop_X^\nu(Z) := \Bend(\trop(\mathscr{I}_Z)).
\]
\end{definition}

\subsection{Basic properties of $\Trop^\nu_X$}
In this section we let $X$ be a locally integral $\Fun$-scheme and $\nu: R\to S$ a valued ring.  We
will present some general properties of tropicalization, but first let us consider the projective case.

Suppose that $X=\PP^n$ and $Z\subset \PP^n_R$ is a closed subscheme over $R$ given by a
homogeneous ideal $I\subset R[x_0, \ldots, x_n]$.  Tropicalization is compatible with the $\proj$ construction
(cf. Theorem \ref{thm:Coxcomp} below) in the following sense:  $\trop(I)$ is a homogeneous ideal in
$S[x_0,\ldots, x_n]$, and $\bend\trop(I)$ is then a homogeneous congruence, which is to say that the
grading descends to the quotient, and one has
\[
\Trop (\proj  R[x_0, \ldots, x_n]/I ) = \proj ( S[x_0,\ldots, x_n] / \bend\trop(I) ).
\]
A key observation (used in defining the Hilbert polynomial and in Theorem \ref{thm:HF2}) is that the
$S$-linear dual,
\[
\left(S[x_0,\ldots, x_n] / \bend\trop(I)\right)_d^\vee,  
\]
of the degree $d$ graded piece of the homogeneous coordinate semiring is a tropical linear space in
$S[x_0,\ldots, x_n]_d^\vee$; namely, it is the tropicalization of the linear subspace
$(R[x_0,\ldots,x_n]/I)_d^\vee \subset R[x_0,\ldots,x_n]_d^\vee$ (as observed in \cite{Maclagan-Rincon}, this
tropical linear space is dual to the tropical linear space $\trop(I_d)$).
This is a special property, and arbitrary projective $\T$-schemes do not satisfy
it in general, as illustrated in the following example.

\begin{example}
Consider the family
\[
Z := \proj\T[x,y,t]/\langle x^2\sim x^2+txy, y^2\sim -\infty \rangle \subset \PP_\T^1 \times \A^1
\rightarrow \A^1,
\]
where $x$ and $y$ are in degree 1, and $t$ is in degree 0 and is the parameter on $\A^1$.  The fiber
$Z_t$ over $t = -\infty$ is cut out by the single equation $y^2 = -\infty$, and it is thus the
tropicalization of the projective scheme cut out by $y^2=0$.  On the other hand, when
$t\neq -\infty$, the degree 2 piece of the homogeneous coordinate semiring of the fiber over $t$
dualizes to the submodule $\{(a,b,-\infty) \:\: | \:\: a\leq b\} \subset \T^3$, which is not a
tropical linear space, and so these fibers cannot be tropicalizations.
\end{example}

This property of projective $\T$-subschemes having homogeneous coordinate semiring whose graded pieces all dualize to tropical linear spaces is an important one.  There is recent and upcoming work investigating this further and developing a theory of such schemes \cite{Maclagan-Rincon,Maclagan-Rincon-Tpts, GGFink}.

Now let $X$ be an arbitrary locally integral $\Fun$-scheme and $\nu: R\to S$ a valued ring.

\begin{proposition}\label{prop:trop-whole-space}
$\Trop_X(X_R) = X_S$.  
\end{proposition}

\begin{proof}
  Consider an affine patch $\spec A \subset X$.  If $I=(0)$ then $\trop(I)=(0)$, and so
  $\bend\trop(I)$ is the trivial congruence.
\end{proof}

One can therefore view the tropical model $X_S$ of $X$ as a canonical tropicalization of $X_R$.
This next lemma is a slight extension of the statement that, when $X$ is a toric variety,
tropicalization commutes with restriction to toric strata and with restriction to open toric subvarieties.

\begin{lemma}\label{lem:restriction-of-trop}
  Let $W\subset X$ be a locally closed locally integral subscheme such that $W$ is locally defined by
  equations of the form $x \sim 0$.  Then $\Trop_X(Z) \cap W_S = \Trop_W (Z\cap W_R)$.  In
  particular, $\Trop_X(W_R) = W_S$.  
\end{lemma}

\begin{proof}
  It suffices to show this in the affine case.  By Lemmas \ref{qc-localization-lemma} and
  \ref{lem:trop-localization-compatibility}, tropicalization commutes with restriction to an open
  subscheme coming from the $\Fun$ level, so we are reduced to the case when $W$ is a closed subscheme, and
  then the result follows from Lemma \ref{lem:monombasis} below.  The equality $\Trop_X(W_R) = W_S$
  then follows from Proposition \ref{prop:trop-whole-space}.
\end{proof}

For $X$ an $\Fun$-scheme and $R$ a (semi)ring, a morphism $\spec R \to X_R$ is given locally by a
multiplicative map from a monoid to $R$.  Thus, a valuation $\nu:R \to S$ determines a map
$\widetilde{\nu}: X(R) \to X(S)$.  In particular, if $X = \mathbb{A}_{\F_1}^n$ then
$\widetilde{\nu}: R^n \to S^n$ is coordinate-wise valuation.

\begin{proposition}
  If $S$ is totally ordered, then the tropicalization of a point $p\in X(R)$ is the image of the
  point under $\widetilde{\nu}$; more precisely, if $Z\subset X_R$ is the closed subscheme
  corresponding to $p$, then $\Trop^\nu_X(Z)$ is the closed subscheme corresponding to the point
  $\widetilde{\nu}(p)\in X(S)$.
\end{proposition}
\begin{proof}
  Locally, $X= \spec A$ with $A$ an integral $\Fun$-algebra; let $\{x_i\}_{i\in \Lambda}$ be a set of generators for the monoid $A$.
  The point $p$ is a multiplicative map $A\to R$ and is thus determined by the collection
  $\{p(x_i)\}_{i\in \Lambda}$ of elements of $R$; the corresponding subscheme $Z$ is defined by the
  ideal $I := (x_i - p(x_i))_{i\in \Lambda}$.  On the other hand, the point $\widetilde{\nu}(p) \in
  X(S)$ is determined by the family of elements $\widetilde{\nu}(p)(x_i) = \nu(p(x_i)) \in S$ and
  corresponds to the congruence $\langle x_i \sim \nu(p(x_i)) \rangle_{i \in \Lambda}$.  

  The ideal $\trop(I)$ (recall Proposition \ref{prop:trop-sends-ideals-to-ideals}) contains the elements $x_i - \nu(p(x_i)) \in A\otimes S$, and hence
  $\bend\trop(I)$ contains the relations $x_i \sim \nu(p(x_i))$ that define the point
  $\widetilde{\nu}(p)$ as a subscheme.  It only remains to show that there are no additional
  relations in $\bend\trop(I)$, i.e., if $f\in I$ then the bend relations of $\nu(f)$ are implied by the relations $x_i \sim \nu(p(x_i))$.  Since $A\otimes S/\langle x_i \sim \nu(p(x_i)) \rangle_{i \in \Lambda} \cong S$, it is equivalent to showing that the map $x_i \mapsto \nu(p(x_i))$ defines an $S$-point of $A\otimes S/\bend\trop(I)$.  Since \[\Hom(A\otimes S/\bend\trop(I), S) = \bigcap_{f\in I} \Hom(A\otimes S/\bend(\nu(f)), S),\] it suffices by Proposition \ref{prop:bend-locus} to show that each $\nu(f)$ tropically vanishes at $\widetilde{\nu}(p)$, in the sense of that proposition.  Since $\nu(f(p)) = \nu(0_R) = 0_S$ but $\nu(f)(\widetilde{\nu}(p)) \ne 0_S$, the valuation does not commute with taking the sum of the terms in $f$, so the result follows from Lemma \ref{lem:val}(2).
\end{proof}

\subsection{Relation to the Kajiwara-Payne extended tropicalization functor}

We now show that the above scheme-theoretic tropicalization recovers the Kajiwara-Payne extended
tropicalization functor \cite{Kajiwara, Payne} upon composition with $\Hom_{\Sch/\T}(\spec \T, -)$.

Let $X$ be a toric variety over $\Fun$, and let $k$ be an algebraically closed field equipped with a
non-trivial valuation $\nu: k\to \T$.  The Kajiwara-Payne extended tropicalization is a map
$\trop_X$ that sends subvarieties of $X_k$ to subsets of $X(\T)$ by sending $Z \subset X_k$ to the
Euclidean closure of the image of $Z(k)$ under the map $\widetilde{\nu}: X(k) \to X(\T)$.  This
extends the usual set-theoretic tropicalization map from tori or affine spaces to toric varieties.

\begin{theorem}\label{thm:KP}
  The set of $\T$-points of $\Trop_X(Z)$ coincides with $\trop_X(Z)$ as a subset of $X(\T)$.
\end{theorem}
\begin{proof}
  By \cite[Prop. 3.4]{Payne}, the set-theoretic tropicalization can be computed stratum by stratum.
  I.e., if $W$ is a torus orbit in $X$ then $\trop_X(Z)\cap W(\T) = \trop_{W}(Z \cap
  W_k)$.  By the Fundamental Theorem of tropical geometry \cite[Theorem 3.2.4]{MS} (a.k.a. Kapranov's
  Theorem in the case of a hypersurface), $\trop_{W}(Z \cap W_k)$ is the subset of points in
  $W(\T)\cong \R^n$ where the graph of each nonzero function in the ideal defining $Z\cap W_k$ is
  nonlinear.  By Proposition \ref{prop:bend-locus} and Lemma \ref{lem:restriction-of-trop}, this is
  equal to the set of $\T$-points of $\Trop_X(Z) \cap W_\T$.
\end{proof}

\subsection{Functoriality of tropicalization}
We now examine the functoriality properties of the scheme-theoretic tropicalization map
$\Trop^\nu_X$.  We show that it is functorial in $X$ in the sense described below, and under certain
additional hypotheses it is functorial in the valuation $\nu$.

For a (semi)ring $R$, let $\mathcal{P}(R)$ denote the category of pairs
\[
(X \text{ a locally integral $\Fun$-scheme}, \:Z\subset X_R \text{ a closed
  subscheme}),
\] 
where a morphism $(X,Z) \to (X', Z')$ is an $\Fun$-morphism $\Phi: X \to X'$ such that $\Phi_R(Z) \subset Z'$.

\begin{proposition}\label{prop:trop-functoriality}
  The tropicalization maps $\{\Trop^\nu_X\}$ determine a functor $\Trop^\nu:
  \mathcal{P}(R) \to \mathcal{P}(S)$ sending $(X,Z)$ to $(X,\Trop^\nu_X(Z))$.
\end{proposition}

\begin{proof}
  Given an arrow $(X,Z)\to (X',Z')$ in $\mathcal{P}(R)$, we must show that $\Phi_S(\Trop_X(Z))\subset
  \Trop_{X'}(Z')$.  It suffices to show this in the affine case: $X=\spec A$, $X'=\spec A'$, the map
  $\Phi$ is given by a monoid homomorphism $\varphi : A' \to A$, and $Z$ and $Z'$ are given by
  ideals $I \subset A\otimes R$ and $I' \subset A'\otimes R$ with $\varphi_R(I')\subset I$.  The claim
  is now that $(\varphi_S)_* \bend\trop(I') \subset \bend\trop(I)$, and for this it suffices to show that $
  (\varphi_S)_*\bend(\nu(f)) \subset \bend(\nu(\varphi_R(f)))$ for any $f\in A'\otimes R$.  In fact, we
  will show that each generating relation
  \begin{equation}\label{rel1}
  \varphi_S(\nu(f)) \sim \varphi_S(\nu(f)_{\widehat{i}})
  \end{equation}
  of $(\varphi_S)_* \bend(\nu(f))$ is implied by the corresponding relation
  \begin{equation}\label{rel2}
  \nu(\varphi_R(f)) \sim \nu(\varphi_R(f))_{\widehat{\varphi(i)}}
  \end{equation}
  in $\bend(\nu(\varphi_R(f)))$ by adding the RHS of \eqref{rel1} to both sides.  We show this by
  comparing coefficients term-by-term.  For $\ell \in \supp(f)$, let $a_\ell \in R$ denote the
  coefficient of $\ell$.  For each $m \in \supp(\varphi_R(f))$ with $m\neq \varphi(i)$, the
  coefficients of $m$ on boths sides of \eqref{rel1} are equal to
\begin{equation}
\label{coefs1}
\sum_{\ell \in \varphi^{-1}(m)} \nu(a_\ell).
\end{equation}
The coefficients of $m$ on either side in \eqref{rel2} are both equal to
\begin{equation}
\label{coefs2}
\nu \left( \sum_{\ell \in \varphi^{-1}(m)} a_\ell \right).
\end{equation}
By the subadditivity property of the valuation, adding \eqref{coefs2} to \eqref{coefs1} yields \eqref{coefs1}.

We now examine the coefficients of $\varphi(i)$ in \eqref{rel1} and
\eqref{rel2}; they are, respectively, 
\begin{align}
\label{coefs3}
\sum_{\ell \in \varphi^{-1}(\varphi(i))} \nu(a_\ell), & &   
\sum_{\ell \in \varphi^{-1}(\varphi(i))  \smallsetminus \{i\}} \nu(a_\ell) \\
\textit{(LHS)} & & \textit{(RHS)} \nonumber
\end{align}
and
\begin{align}
\label{coefs4}
\nu \left( \sum_{\ell \in \varphi^{-1}(\varphi(i))} a_\ell \right) & &
0_S.\\ \textit{(LHS)} & & \textit{(RHS)} \nonumber
\end{align}
By Lemma \ref{lem:val} part (1), adding the RHS of \eqref{coefs3} to both sides of
\eqref{coefs4} yields \eqref{coefs3}.
\end{proof}

\begin{remark}\label{rem:sign-required}
The sign condition in Definition \ref{def:valuation} is necessary for the above proposition to
hold.  For example, suppose $f= \ell_1 - \ell_2$, $\varphi$ maps the two monomials
$\ell_1$ and $\ell_2$ in the support of $f$ to the same monomial.  When $i=\ell_2$ Then \eqref{coefs3} becomes
$\nu(1_R) + \nu(-1_R) \sim \nu(1_R)$, and \eqref{coefs4} is $\nu(0_R) \sim 0_s$.  On the other hand,
taking $i=\ell_1$, \eqref{coefs3} now gives $\nu(1_R) + \nu(-1_R) \sim \nu(-1_R)$.  Thus it must be
the case that $\nu(-1_R)=\nu(1_R)$.
\end{remark}

We now turn to the dependence on $\nu$.
\begin{proposition}\label{prop:val-functoriality}
 Let $\nu: R\to S$ be a valuation and $\varphi: S \to T$ a map of semirings.  Then 
\[
\Trop_X^{\varphi\circ\nu} (Z) = \Trop_X^{\nu}(Z) \times_{\spec S} \spec T
\]
as subschemes of $X_{T}$.
\end{proposition}
\begin{proof}
  It suffices to prove this in the case $X$ is affine, so assume $X=\spec A$ for some integral
  $\Fun$-algebra $A$, and let $I\subset A\otimes R$ be the ideal defining $Z\subset X$.  Given a
  module congruence $K$ on $A \otimes S$, the canonical isomorphism $(A\otimes S)\otimes_S T \cong A
  \otimes T$ descends to an isomorphism
\[
\big( A\otimes S / K \big) \otimes_S T \cong A\otimes T / \varphi_* K.
\]
The claim follows from this by taking $K=\bend\trop^\nu(I)$ and observing that
$\varphi_*\bend\trop^\nu(I) = \bend\trop^{\varphi\circ \nu}(I)$.
\end{proof}

\subsection{Moduli of valuations and families of tropicalizations}\label{sec:moduli}

Let $\val(R) := \spec S^R_{univ}$ be the affine $\B$-scheme corresponding to the semiring
of values associated with the universal valuation on $R$ defined in \S\ref{sec:valuations}.
By Proposition \ref{val-univ-property}, $\val(R)$ represents the functor on affine $\B$-schemes,
\[
\spec S \mapsto \{\text{valuations $R\to S$}\}.
\]
Thus $\val(R)$ is the moduli scheme of valuations on $R$.  This is a refinement of the observation
of Manon \cite{Manon} that the set of all valuations with semiring of values $\T$ forms a fan.  In
particular, the $\T$-points of $\val(R)$ are the usual non-archimedean valuations on $R$.

As a special case of Proposition \ref{prop:val-functoriality} we have the following (Theorem
\ref{thm:ValFamily} part (1) from the introduction).
\begin{theorem}\label{thm:valfamily1}
  Given a locally integral $\Fun$-scheme $X$, a ring $R$, and a subscheme $Z\subset X_R$, the
  tropicalization of $Z$ with respect to the universal valuation, $\Trop^{\nu^R_{univ}}_X(Z)$, forms
  an algebraic family of $\B$-schemes over $\val(R)$ such that the fiber over each valuation $\nu$
  is $\Trop_X^\nu(Z)$.
\end{theorem}

\subsection{Compatibility with Cox's quotient construction}

Let $X=X^\Delta$ be a toric scheme over $\F_1$ and recall (\S\ref{sec:Cox}) that $X = U/G$, where
$U\subset \mathbb{A}^{\Delta(1)}$ is the complement of the vanishing of the irrelevant ideal and $G
= \Hom(\Cl(X),\Z)$.  A homogeneous ideal $I \subset \Cox(X_R) = R[x_\rho~|~\rho\in\Delta(1)]$
determines a closed subscheme $Z \subset X_R$, and if $\Delta$ is simplicial then every closed
subscheme arises in this way \cite[Theorem 3.7]{Cox}.  The scheme $Z$ is the categorical quotient of
the $G$-invariant locally closed subscheme $\widetilde{Z} \cap U_R \subset
\mathbb{A}^{\Delta(1)}_R$, where $\widetilde{Z} := V(I)$.  In other words, we have
\[
Z  = (\widetilde{Z} \cap U_R)/G \subset U_R/G = X_R.
\]

\begin{theorem}\label{thm:Coxcomp}
Tropicalization commutes with the Cox quotient: \[\Trop_X(Z) = (\Trop_{\mathbb{A}^{\Delta(1)}}(\widetilde{Z}) \cap U_S)/G \subset U_S/G = X_S.\] 
\end{theorem}

\begin{proof}
  In the notation of \S\ref{sec:CoxHypersurf}, we can cover $X$ by open affines $X_\sigma$ for
  $\sigma\in \Delta$.  The subscheme $Z \subset X_R$ is defined in each such chart as $Z_\sigma :=
  \spec (\Cox(X_R)[{x^{-1}_{\widehat{\sigma}}}]_0/I'_0)$, where $I'$ denotes the image of $I$ in
  this localization and $I'_0$ its degree zero part.  The tropicalization $\Trop_X(Z)$ is then
  obtained by gluing the affine tropicalizations
  \[
  \Trop_{X_\sigma}(Z_\sigma) = \spec \Cox(X_S)[{x^{-1}_{\widehat{\sigma}}}]_0/\bend\trop(I'_0). 
  \]
  Since the valuation preserves degree and taking quotients commutes with taking degree zero part,
  this is the spectrum of $\left( \Cox(X_S)[{x^{-1}_{\widehat{\sigma}}}]/\bend\trop(I') \right)_0$.  As in
  \S\ref{sec:Cox}, taking degree zero here coincides with taking the subalgebra of $G$-invariants,
  and by Lemma \ref{lem:restriction-of-trop} tropicalization commutes with $\F_1$-localization, so this
  is the categorical quotient of $\Trop_{\mathbb{A}^{\Delta(1)}}(\widetilde{Z}) \setminus
  V(x_{\widehat{\sigma}})$.  In the usual way, these categorical quotients patch together to yield
  the categorical quotient of $\Trop_{\mathbb{A}^{\Delta(1)}}(\widetilde{Z}) \cap U_S$.
\end{proof}

\section{Numerical invariants}\label{sec:numerical-invariants}
Here we show that there is a natural way to define Hilbert polynomials for the class of projective
subschemes over idempotent semirings that arise as tropicalizations, and that tropicalization
preserves the Hilbert polynomial.  We also show that for a projective hypersurface, the
multiplicities (sometimes called weights) decorating the facets of its tropicalization, which are
frequently used in tropical intersection theory, are encoded in the tropical scheme structure.

\subsection{The Hilbert polynomial}

First recall the classical setup.  Let $k$ be a field, $A := \Fun[x_0,\ldots, x_n]$ the homogeneous
coordinate algebra of $\PP^n_{\Fun}$, and $Z\subset \PP^n_k$ a subscheme defined by a homogeneous
ideal $I\subset A\otimes k$.  The Hilbert function of $A\otimes k/I$ is usually defined to be the
map $d \mapsto \dim_k(A\otimes k/I)_d$; however, one could equally well replace
$\dim_k(A\otimes k/I)_d$ with $\dim_k(A\otimes k/I)^\vee_d$, where $(-)^\vee$ is the $k$-linear
dual.  This may seem like a trivial observation since a finite dimensional vector space and its dual
are noncanonically isomorphic, but the choice between the two becomes contenful when we come to
define the Hilbert function over idempotent semirings (cf. Theorem \ref{thm:double-dual}).  All
homogeneous ideals defining $Z$ have the same saturation, so the corresponding Hilbert functions
coincide for $d \gg 0$ and this determines the Hilbert polynomial of $Z \subset \PP_k^n$.

To define the Hilbert function for a homogeneous congruence $K$ on $A\otimes S$, where $S$ is an
idempotent semiring, one first needs an appropriate definition of the dimension of an $S$-module.
We assume here $S$ is a totally ordered semifield.  The following definition is from
\cite{Mikhalkin-Zharkov}, in the case $S=\T$.
\begin{definition}\label{def:lindim}
  Let $S$ be a totally ordered semifield and $L$ an $S$-module.  
\begin{enumerate}
\item A collection $v_1, \dots, v_k \in L$ is \emph{linearly dependent} if any linear combination of
  the $v_i$ can be written as a linear combination of a proper subset of the $v_i$; otherwise it is
  \emph{linearly independent}.
\item The \emph{dimension} of $L$, denoted $\dim_S L$, is the largest number $d$ such that there
  exists a set of $d$ linearly independent elements in $L$.
\end{enumerate}
\end{definition}

We now show that the above notion of dimension is preserved under base change.

\begin{lemma}\label{lem:dim-base-change}
  Let $\varphi: S \to T$ be a homomorphism of totally ordered idempotent semifields.  If $L$ is a
  submodule of a finitely generated free $S$-module then $\dim_S L = \dim_T L\otimes_S T$.
\end{lemma}
\begin{proof}
  It suffices to prove the result when $T = \B$ and $\varphi$ is the unique homomorphism to $\B$,
  defined by sending all nonzero elements to $1_\B$.  Moreover, since $L$ is a submodule of a
  finitely generated free module, it suffices to show that a set $v_1, \ldots, v_d \in S^n$ is
  linearly independent (in the sense of the above definition) if and only if the set $\varphi(v_1), \ldots,
  \varphi(v_d)$ is linearly independent.  Clearly if the $v_i$ are $S$-linearly dependent then
  their images under $\varphi$ are $\B$-linearly dependent.  Conversely, suppose $\varphi(v_1),
  \ldots, \varphi(v_d)$ are $\B$-linearly dependent, so that (without loss of generality)
  \[
  \sum^d_{i = 1} \varphi(v_i) = \sum_{i=1}^{d-1} \varphi(v_i).
  \]
  This condition says that, for each $k$, if the $k^{th}$ component of $v_i$ vanishes for $i < d$
  then it does so for $v_d$ as well.  Since $S$ is a totally ordered semifield, given any
  notrivial elements $a,b\in S$, there exists $c\in S$ such that $ca \geq b$.  Hence for each $i <
  d$, we may choose an $a_i \in S$ large enough so that each component of $a_i v_i$ is greater than
  or equal to the corresponding component of each $v_d$.  By construction we then have
  \[
  \sum_{i=1}^d a_i v_i = \sum_{i=1}^{d-1} a_i v_i,
  \]
  which shows that the $v_i$ are $S$-linearly dependent.
\end{proof}

\begin{lemma}\label{lem:lin-space-dim}
  If $L\subset S^n$ is a tropical linear space of rank $r$ (in the sense of \cite{Frenk} or
  \cite{Speyer1, Speyer2}) then $\dim_S L = r$.
\end{lemma}
\begin{proof}
  Let $\psi: S \to \T$ be any homomorphism (for example, one can take the unique homomorphism to
  $\B$ followed by the unique homomorphism $\B \hookrightarrow \T$).  The base change $L\otimes_S\T$
  is a tropical linear space of rank $d$ in $\T^n$ (this can easily be seen in terms of the
  corresponding valuated matroids).  By Lemma \ref{lem:dim-base-change}, $\dim_S L = \dim_\T
  L\otimes_S \T$ and by \cite[Proposition 2.5]{Mikhalkin-Zharkov}, $\dim_\T L\otimes_S \T$ is equal
  to the maximum of the local topological dimensions of the polyhedral set underlying $L\otimes_S
  \T$.  The statement now follows from the fact that a tropical linear space in $\T^n$ is a
  polyhedral complex of pure dimension equal to its rank.
\end{proof}

\begin{definition}
  Given a homogenous congruence $K$ on $A\otimes S = S[x_0, \ldots, x_n]$, the \emph{Hilbert function} of $K$ is the map $d \mapsto \dim_S (A\otimes S/K)^\vee_d$.
\end{definition}

Two homogeneous congruences (cf. \S\ref{sec:CoxHypersurf}) define the same projective subscheme if
and only if they coincide in all sufficiently large degrees.  Given a homogeneous congruence $K$,
the \emph{saturation} $K^{sat}$ is the maximal homogeneous congruence that agrees with $K$ in
sufficiently high degrees (this exists since the sum of any two congruences that coincide with $K$
in high degrees will itself coincide with $K$ in high degrees), and we say that $K$ is
\emph{saturated} if $K=K^{sat}$.

\begin{definition}
The Hilbert function of a projective subscheme $Z\subset \PP^n_S$ is the Hilbert function of the
unique saturated homogeneous congruence on $S[x_0, \ldots, x_n]$ defining $Z$.
\end{definition}

Since modules over a semiring do not form an abelian category, it does not appear automatic that the
Hilbert function of an arbitrary projective subscheme over $S$ is eventually polynomial, but
remarkably, this is the case for schemes in the image of the tropicalization functor.

\begin{theorem}\label{thm:HF2}
  Let $\nu: k \to S$ be a valued field.  If $I \subset A\otimes k$ is a homogenous ideal then the
  Hilbert function of $I$ coincides with the Hilbert function of $\bend\trop(I)$.  Consequently, for
  any subscheme $Z\subset \PP_k^n$, the tropicalization $\Trop^{\nu}_{\PP^n}(Z)\subset \PP_S^n$ has
  a well-defined Hilbert polynomial and it coincides with that of $Z$.
\end{theorem}

\begin{proof}
  The operation $\bend\trop(-)$ commutes with restriction to the degree $d$ graded piece, so
  \[
  \left(A\otimes S / \bend\trop(I) \right)_d = A_d \otimes S / \bend\trop(I_d).
  \]
  By Propositions \ref{prop:bend-locus} and \ref{prop:intersection-of-bend-loci}, the dual,
  $\left(A_d \otimes S/\bend\trop(I_d) \right)^\vee$, is the tropical linear space in $(A_d\otimes
  S)^\vee$ that is the tropicalization of the linear subspace $(A_d \otimes k / I_d)^\vee \subset
  (A_d \otimes k)^\vee$.  Since the tropicalization of a subspace of dimension $r$ is a rank $r$
  tropical linear space, the statement that $\bend\trop(-)$ preserves the Hilbert function of $I$
  now follows from Lemma \ref{lem:lin-space-dim}.  The statement about the Hilbert polynomials then
  follows since, by Theorem \ref{thm:Coxcomp}, $\Trop^{\nu}_{\PP^n}(Z)$ is defined by the
  homogeneous congruence $\bend\trop(I)$.
\end{proof}

Recall that, classically, a family of subschemes in projective space is flat if and only if the
Hilbert polynomials of the fibers are all equal.  The above result therefore suggests that if one
views tropicalization as some kind of degeneration of a variety, then the numerical behavior is that
of a flat degeneration.  Moreover, this next result (Theorem \ref{thm:ValFamily} part (2)) shows
that the family of all tropicalizations of a projective subscheme $Z$ has the numerical behaviour of
a flat family.

\begin{corollary}
  For $S$ a totally ordered idempotent semifield, the Hilbert polynomial of the fiber of the family
  $\Trop^{\nu^k_{univ}}(Z) \to \val(k)$ over any $S$-point is equal to the Hilbert polynomial of $Z$.
\end{corollary}
\begin{proof}
This follows directly from Theorems \ref{thm:valfamily1} and \ref{thm:HF2} since the Hilbert
polynomials of the fibers are all equal to the Hilbert polynomial of $Z$.
\end{proof}

\subsection{Recovering the multiplicities and the defining polynomial of a tropical hypersurface}\label{sec:mults}

\begin{proposition}\label{prop:trophyprecov}
  For any valued ring $\nu : R \rightarrow S$ such that $S$ is a semifield, and any projective
  hypersurface $Z = V(f) \subset \PP_R^n$, the tropicalized scheme $\Trop(Z) \subset \PP_S^n$ determines
  the defining homogeneous polynomial $\nu(f)\in (A\otimes S)_d$ uniquely up to scalar.
\end{proposition}

\begin{proof}
  Since $\Trop(Z) = \proj A\otimes S/\langle \bend(\nu(g))\rangle_{g\in (f)}$, and this homogeneous
  congruence in degree $d$ coincides with the congruence $\bend(\nu(f))$, the result follows from
  Lemma \ref{lem:hyprecov}.
\end{proof}

\begin{corollary}\label{cor:multiplicities}
  For an algebraically closed valued field $\nu : k \rightarrow \T$, and an irreducible projective hypersurface $Z \subset \PP_k^n$ that is not contained in any coordinate hyperplane, the scheme
  $\Trop(Z)\subset\PP_\T^n$ determines the multiplicities on the facets of the restriction of its
  $\T$-points to the tropical torus $\mathbb{R}^n$.\footnote{Maclagan and Rinc\'on have now extended this result
    from hypersurfaces to arbitrary irreducible projective varieties \cite{Maclagan-Rincon}.}
\end{corollary}

\begin{proof}
  This follows immediately from Proposition \ref{prop:trophyprecov}, since the multiplicities for a
  tropical hypersurface are lattice lengths in the Newton polytope of $f$ \cite[\S2]{tropdiscrim}.
\end{proof}

\section{Hypersurfaces, bend loci, and tropical bases}\label{sec:hypbases}

Let $\nu: R \to S$ be a valued ring and $f\in R[x_1, \ldots, x_n]$.  This data gives rise to two
subschemes of $\mathbb{A}^n_S$ that are distinct in general.  The first is the bend locus
$\Bend(\nu(f))$ of the coefficient-wise valuation of $f$ --- it is cut out by the bend relations of
the single polynomial $\nu(f)$.  The second is the tropicalization $\Trop(V(f))$ of the classical
hypersurface determined by $f$ --- it is cut out by the bend relations of the coefficient-wise
valuations of all elements in the principal ideal $(f)$.

In this section we compare $\Bend(\nu(f))$ and $\Trop(V(f))$ as schemes. The $S$-points of each
coincide (at least when $S=\T$ --- see \cite[Example 2.5.5]{MS}), so the set-theoretic bend locus of
$\nu(f)$ is equal to the set-theoretic tropicalization of $V(f)$; in the literature this set is
usually referred to as the tropical hypersurface of $\nu(f)$.  However, as schemes they are
generally different, though they do sometimes agree, such as when $f$ is a monomial or binomial (see
Proposition \ref{prop:trop-basis}).

If $R$ is a field then the discrepancy between these two schemes can be understood in terms of
Theorem \ref{thm:HF2}: the tropicalization of a projective hypersurface must have enough relations
in its homogeneous coordinate algebra to yield the Hilbert polynomial of a codimension one
subscheme, but the bend relations of a single tropical polynomial do not typically suffice for this
numerical constraint.

This discussion leads naturally to the notion of a scheme-theoretic tropical basis, a term we
introduce as a replacement for the usual set-theoretic notion considered in the tropical literature
(e.g., \cite[\S2.5]{MS}).

\subsection{Bend loci versus tropical hypersurfaces}\label{sec:vs}

Let $A$ be an integral $\Fun$-algebra, $\nu: R \to S$ a valued ring, and $f\in A\otimes R$.  With a
slight abuse of notation, we will write $\trop(f)$ for $\trop(-)$ applied to the principal ideal
$(f)$.  We are concerned with comparing the congruences $\langle \bend(\nu(f))\rangle$ and
$\bend\trop(f)$.  The former defines the bend locus $\Bend(\nu(f))$, while the latter defines the
tropicalization of the hypersurface $V(f)$.  To illustrate that $\Trop(V(f))$ can be strictly smaller
than $\Bend(\nu(f))$ consider the following example.

\begin{example}\label{ex:tropbas}
  Let $R$ be a ring equipped with the trivial valuation $\nu: R\to \B$, and let
  $f= x^2 + xy + y^2 \in R[x,y]$.  One can see as follows that the congruence $\bend\trop(f)$ is
  strictly larger than the semiring congruence $\langle \bend(\nu(f)) \rangle$.  This latter
  congruence is generated by the degree 2 relations $x^2 + y^2 \sim x^2 + xy \sim xy + y^2$.  The
  degree 3 part of $\langle\bend(\nu(f))\rangle$ is generated (as a module congruence) by the
  relations $\bend(x^3 + x^2 y + xy^2)$ and $\bend(x^2 y + xy^2 + y^3)$. If $g,h$ are polynomials in
  $\B[x,y]$, then $gh$ and $g+h$ each have at least as many monomial terms as $g$, and from this
  observation it follows that any nontrivial degree 3 relation in $\langle\bend(\nu(f))\rangle$
  involves only polynomials with at least 2 terms.  However, $(x-y)f = x^3 - y^3$, and this gives
  the degree 3 monomial relation $x^3 \sim y^3$ in $\bend\trop(f)$.  This behavior, where
  $\bend(\nu(f))$ does not generate all the relations $\bend\trop(f)$, appears to be generic.
  Suppose now that $f = x^2 + xy+ ty^2$ for some $t \neq 0,1$.  The degree 3 part of
  $\langle\bend(\nu(f))\rangle$ is generated as a module congruence by the bend relations of
  $\nu(xf) = x^3 + x^2 y + \nu(t) xy^2$ and $\nu(yf)=x^2 y + xy^2 + \nu(t) y^3$.  However, in
  $\bend\trop(f)$ one also has the bend relations of
  $\nu( (x-ty)f) = x^3 + \nu(1-t)x^2 y + \nu(t)^2 y^3$; among these is the relation
  \[
  x^3 + \nu(t)^2 y^3 \sim x^3 + \nu(1-t) x^2 y
  \]
  which cannot be obtained from $\bend(\nu(xf))$ and $\bend(\nu(yf))$. In fact, one can check that
  these relations now generate all relations in the degree 3 part of $\bend\trop(f)$.
\end{example}

In general, when passing from $\langle \bend(\nu(f)) \rangle$ to $\bend\trop(f)$, the additional
relations appearing in $\bend\trop(f)$ are not uniquely determined by the single tropical
polynomial $\nu(f)$, so the tropicalization of a hypersurface is not uniquely determined by the bend
locus of the valuation of a defining polynomial.  The following is a simple example illustrating
this: two polynomials with the same valuation but whose associated hypersurfaces have distinct
tropicalizations as schemes.

\begin{example}
  Let $k = \mathbb{C}$ with the trivial valuation $\nu : k \rightarrow \B$, and consider the
  polynomials in $\C[x,y]$,
  \[
  f=a_1x^2+a_2xy+a_3y^2, \text{ and } g = x^2+xy+y^2,
  \]
  where the coefficients in $f$ do not satisfy the quadratic relation $a_2^2=a_1a_3$.  Clearly
  $\nu(f)=\nu(g)$, and as seen in Example \ref{ex:tropbas}, $\bend\trop(g)$ contains the relation
  $x^3 \sim y^3$. However, for any nonzero linear form $h = b_1 x + b_2 y \in \mathbb{C}[x,y]$ the
  polynomial $fh$ has at least three terms, so $\bend\trop(f)$ cannot contain the relation $x^3
  \sim y^3$.
\end{example}

There are, however, certain nice situations where the tropicalization of an ideal is equal to the
intersection of the bend loci of a set of generators of the ideal.

\begin{proposition}\label{prop:trop-basis}
  Let $A$ be a torsion-free integral monoid-with-zero, and suppose $S$ is totally ordered.  If $f =
  ax + by$ is a binomial ($a,b \in R$, and $x,y \in A$) then $\bend\trop(f) =
  \langle\bend(\nu(f))\rangle$.
\end{proposition}
\begin{proof}
  We must show that $\langle\bend(\nu(f))\rangle$ implies $\langle\bend(\nu(fg))\rangle$ for any $g\in A\otimes R$.  Since $f$
  is a binomial, $\langle\bend(\nu(f))\rangle$ is generated by the single relation $\nu(a)x \sim \nu(b)y$.

  We define a binary relation `$\to$' on the set $\supp(g)$ as follows: $z_1 \to z_2$ if
  $z_1 x = z_2 y$.  This generates an equivalence relation; let $\{C_i\}$ be the set of equivalence
  classes.  Note that $C_i x \cup C_i y$ is necessarily disjoint from $C_j x\cup C_j y$ if
  $i\neq j$.  Hence we can, without loss of generality, assume that $\supp(g)$ consists of just a
  single equivalence class $C$. If $C$ consists of a single element then the claim holds trivially,
  so we assume that $C$ consists of at least 2 elements.

  Since $A$ is integral and torsion free, $C$ must consist of a sequence of elements $z_1, \ldots,
  z_n$ such that $z_i x = z_{i+1} y$ (having a loop would imply that $xy^{-1}$ is a torsion element
  in the group completion of $A$, and the integral condition implies that if $x\to y$ and $x\to y'$
  then $y=y'$).

  Let $c_i$ be the coefficient of $z_i$ in $g$.  We then have
  \[
  \nu(fg) = \nu(ac_n) z_n x + \nu (ac_{n-1} + bc_n)z_{n-1}x + \cdots + \nu(ac_1 + bc_2)z_1 x +
  \nu(bc_1)z_1 y.
  \]
  We first show that the relation $\nu(a)x \sim \nu(b)y$ allows the first term, $\nu(ac_n) z_n x$,
  to be absorbed into one of the terms to its right.  First,
  \[
  \nu(ac_n) z_n x \sim \nu(bc_n)z_{n}y = \nu(bc_n)z_{n-1}x.
  \]
  Either $\nu (ac_{n-1} + bc_n) = \nu (ac_{n-1}) + \nu(bc_n)$, in which case we are done, or
  $\nu(ac_{n-1}) = \nu(bc_n)$, in which case
  $\nu(bc_n)z_{n-1}x = \nu(ac_{n-1})z_{n-1}x \sim \nu(b c_{n-1})z_{n-2}x$.  We continue in this
  fashion until either the term $\nu(ac_n) z_n x$ absorbs or we reach the end of the chain, at which
  point it will be absorbed into the final term $\nu(bc_1)z_1 y$.  Working from right to left
  instead, the final term can be absorbed into the terms to its left by the same argument.

  Finally, given a middle term, $\nu (ac_{i-1} + bc_i)z_{i-1}x$, we have that $\nu
  (ac_{i-1})z_{i-1}x$ and $\nu(bc_i)z_{i-1}x$ are both larger, and so the above argument in reverse
  allows us to replace the term $\nu (ac_{i-1} + bc_i)z_{i-1}x$ with $\nu (ac_{i-1})z_{i-1}x +
  \nu(bc_i)z_{i-1}x$.  Then the above argument in the forward direction allows these two terms to be
  absorbed into the terms to the right and left respectively.
\end{proof}

\begin{lemma}\label{lem:monombasis}
  Let $A$ be an integral $\Fun$-algebra, $\nu:R \to S$ a valued ring, and $I$ an ideal in
  $A\otimes R$ generated by elements $f_1, \ldots, f_n$.  If
  $\bend\trop(I) = \langle \bend(\nu(f_i)) \rangle_{i=1\ldots n}$, and $J$ is the ideal generated by
  $I$ together with a monomial $f_0\in A$, then
  \[\bend\trop(J) = \langle \bend(\nu(f_i)) \rangle_{i=0\ldots n}.\]
\end{lemma}
\begin{proof}
  We will show that the generating relations of $\bend\trop(J)$ are all contained in the
  sub-congruence $\langle \bend(\nu(f_i))\rangle_{i=0\ldots n}$.  Let
  $g = \sum_{i=0}^n h_i f_i \in J$, with $h_i \in A\otimes R$.  Since
  $\langle\bend(f_0)\rangle = \langle f_0 \sim 0_S \rangle$, for any $F \in A\otimes S$, the
  congruence $\langle\bend(f_0)\rangle$ contains the relation $F \cdot f_0 \sim 0_S$.  This means
  that if $F, F' \in A\otimes S$ have identical coefficients away from the set of monomials
  $f_0 \cdot A$, then the relation $F \sim F'$ is contained in $\langle\bend(f_0)\rangle$.
  
  Consider the polynomials $F = \nu(g)$ and $F' = \nu(\sum_{i=1}^n h_i f_i)$ in $A\otimes S$; they
  dffer only outside of $f_0 \cdot A$, as do $F_{\widehat{j}}$ and $F'_{\widehat{j}}$.  In
  $\langle \bend(\nu(f_i))\rangle_{i=0\ldots n}$ we thus have the relations
\begin{align*}
F &\sim F'  & \text{from } \bend(f_0) \\
   & \sim F'_{\widehat{j}} & \text{from } \bend\trop(I) = \langle \bend(\nu(f_i)) \rangle_{f = 1\ldots n} \\
  & \sim F_{\widehat{j}} & \text{from } \bend(f_0), \\
\end{align*}
This completes the proof.
\end{proof}

If $f_0$ is instead a binomial then the analogue of the above lemma can fail.

\begin{example}
  Consider $f_1=x-y$ and $f_0=x+y$ in $R[x,y]$ and the trivial valuation $\nu : R \rightarrow \B$.
  By Proposition \ref{prop:trop-basis}, $\bend\trop(f_1) = \langle\bend(\nu(f_1))\rangle$.  However, $\langle
  \bend(\nu(f_0)), \bend(\nu(f_1))\rangle = \langle x \sim y \rangle$ is not the tropicalization of
  the ideal $(f_0,f_1)$, since the latter contains the bend relation of $\nu(f_0+f_1)=x$, namely
  $\langle x \sim -\infty \rangle$, which is not implied by the former.
\end{example}

\subsection{Tropical bases}\label{sec:trop-bases}

It is well-known that the set-theoretic tropicalization of the variety defined by an ideal is not
necessarily equal to the intersection of the (set-theoretic) tropical hypersurfaces associated with
a set of generators of this ideal.  A set of generators for which this holds is called a
\emph{tropical basis} in \cite[\S2.5]{MS}, where this notion is studied and related to Gr\"obner theory.
We use the term \emph{set-theoretic tropical basis} for this concept to distinguish it from the
following notion of tropical basis that arises when considering scheme-theoretic tropicalization.

\begin{definition}
Let $\nu: R \to S$ be a valued ring, $X$ a locally integral $\Fun$-scheme, and $Z\subset X_R$ a closed
subscheme.  A \emph{scheme-theoretic tropical basis} for $Z$ is a set $\{Y_1, Y_2, \ldots\}$
of hypersurfaces in $X_R$ containing $Z$ such that the following scheme-theoretic intersections hold:
\[
Z = \bigcap_i Y_i \text{\:\: and \:\:} \Trop(Z) = \bigcap_i \Trop(Y_i).
\]
\end{definition}

In the affine case, say $X = \spec A$ and $Z = \spec A\otimes R/I$, a scheme-theoretic tropical
basis is a generating set $\{f_1,f_2,\ldots\}$ for the ideal $I$ such that the corresponding
congruences $\bend\trop(f_i)$, obtained by tropicalizing the principal ideals $(f_i)$, generate the
congruence $\bend\trop(I)$.  Note that this is generally a weaker requirement than the requirement
that the bend relations of the $f_i$ generate $\bend\trop(I)$.  For instance, for a principal ideal
$I = (f)$ it is automatic that $\{f\}$ is a scheme-theoretic tropical basis, whereas it is not
always the case, as discussed above in \S\ref{sec:vs} and Example \ref{ex:tropbas}, that
$\bend\trop(I) = \langle\bend(\nu(f))\rangle$.

Not surprisingly, being a scheme-theoretic tropical basis is a stronger requirement than being a
set-theoretic tropical basis.

\begin{example}
  Let $R=k[x,y,z]$ with the trivial valuation.  As discussed in \cite[Example 3.2.2]{MS}, the
  elements $x+y+z$ and $x+2y$ do not form a tropical basis for the ideal $I$ they generate, since
  $y-z\in I$ tropically yields the relation $y\sim z$ which is not contained in $\langle
  \bend(x+y+z),\bend(x+y)\rangle$.  This can be rectified by adding the element $y-z$, and indeed
  these three polynomials form a set-theoretic tropical basis for $I$.  However, if we instead add
  the element $(y-z)^2\in I$ then the corresponding congruence has the same $\T$-points, so this is
  still a set-theoretic tropical basis, but it is no longer a scheme-theoretic tropical basis since
  the relation $y\sim z$ is still missing.
\end{example}

\begin{remark}
  It is known that subvarieties of affine space defined over an algebraically closed field with
  non-trivial valuation admit \emph{finite} set-theoretic tropical bases (see \cite[Corollary
  2.3]{Speyer1} and \cite[Corollary 3.2.3]{MS}).  It would be interesting to see if this also holds
  scheme-theoretically.
\end{remark}

\section*{Table of notation}

\small
\begin{center}
\begin{tabular}{|c|c|c|}
\hline
Symbol & Description & Reference\\
\hline
$\ker\varphi$ & the congruence kernel $S\times_R S$ of a semiring homomorphism $\varphi : S \rightarrow R$ & Definition \ref{def:congker}\\
$\langle J \rangle$ & semiring congruence on $S$ generated by a set of pairs $J \subset S\times S$ & Lemma \ref{lem:CongGen}\\ 
$\bend(f)$ & module congruence of ``bend relations'' of a tropical polynomial $f$ & Definition \ref{def:setbend}\\ 
$\setBend(f)$ & set-theoretic bend locus (``tropical hypersurface'') of a tropical polynomial $f$ & Definition \ref{def:setbend}\\ 
$\Bend(f)$ & scheme-theoretic bend locus of a tropical polynomial $f$ & Definition \ref{def:bendideal}\\ 
$\nu(f)$ & tropical polynomial obtained by coefficient-wise valuation of a polynomial $f$ & \S\ref{sec:tropconstruct}\\
$\trop(I)$ & tropical ideal associated to an ideal $I$ in a valued ring& Proposition \ref{prop:trop-sends-ideals-to-ideals}\\
$\trop(Z)$ & set-theoretic tropicalization of a subvariety $Z$ & Eq. \eqref{eqn:settrop}\\
$\Trop(Z)$ & scheme-theoretic tropicalization of a subscheme $Z$ & Definition \ref{def:affinetrop}\\ 
$\dim_S M$ & maximal cardinality of an independent set in an $S$-module $M$ & Definition \ref{def:lindim}\\
$M^\vee$ & $S$-linear dual $\Hom_S(M,S)$ of an $S$-module & \S\ref{sec:unicity}\\ 
$\nu_{univ}^R$ & universal valuation on a ring $R$, with values in semiring $S_{univ}^R$ & \S\ref{sec:valuations}\\
$\val(R)$ & moduli space of valuations on a ring $R$ & \S\ref{sec:moduli}\\
\hline
\end{tabular}
\end{center}
\normalsize

\bibliographystyle{amsalpha}
\bibliography{bib}
\end{document}